\documentclass[11pt]{amsart}
\usepackage{amssymb,amsfonts,amsthm,amsmath,mathrsfs,xspace,hyperref}
\usepackage{graphicx}
\usepackage[francais,english]{babel}
\usepackage{inputenc}
\usepackage[T1]{fontenc}
\usepackage[centering]{geometry}
\usepackage{csquotes}
\usepackage{color}
\usepackage{enumerate}
\usepackage{enumitem}
\usepackage{stmaryrd}

  \newcommand{\R}{\ensuremath{\mathbb{R}}}%
  \newcommand{\Z}{\ensuremath{\mathbb{Z}}}%
  \newcommand{\N}{\ensuremath{\mathbb{N}}}%
                \newcommand{\Sbb}{\ensuremath{\mathbb{S}}}%

  \newcommand{\Sym}{\ensuremath{\operatorname{Sym}}}%
    \newcommand{\sym}{\ensuremath{\operatorname{Sym}}}%
    \newcommand{\alt}{\ensuremath{\operatorname{Alt}}}%

  \newcommand{\homeo}{\ensuremath{\operatorname{Homeo}}}%
    \newcommand{\aut}{\ensuremath{\operatorname{Aut}}}%
        \newcommand{\psl}{\ensuremath{\operatorname{PSL}}}%

\newcommand{\T}{\mathcal{T}}

	\newcommand{\autT}{\mathrm{Aut}(\T)}

\theoremstyle{definition}
  \newtheorem{defin}{Definition}[section]

  \newtheorem{question}[defin]{Question}

\theoremstyle{plain}
  \newtheorem{thm}[defin]{Theorem}
  \newtheorem{main thm}{Theorem}
  \newtheorem{prop}[defin]{Proposition}
    \newtheorem{prop-def}[defin]{Proposition-Definition}

  \newtheorem{cor}[defin]{Corollary}
   \newtheorem{lem}[defin]{Lemma}

\theoremstyle{remark}

\begin{document}

  \date{\today}	
\title{Triple transitivity and non-free actions in dimension one}

\author{Adrien Le Boudec}
\address{CNRS, UMPA - ENS Lyon, 46 all\'ee d'Italie, 69364 Lyon, France}
\email{adrien.le-boudec@ens-lyon.fr}

\author{Nicol\'as Matte Bon}
\address{CNRS,
	Institut Camille Jordan (ICJ, UMR CNRS 5208),
	Universit\'e de Lyon,
	43 blvd.\ du 11 novembre 1918,	69622 Villeurbanne,	France}
\email{mattebon@math.univ-lyon1.fr}

\thanks{The authors were supported by ANR-14-CE25-0004 GAMME and ALB was supported by a PEPS grant from CNRS}

\maketitle


\begin{abstract}
	The transitivity degree of a group $G$ is the supremum of all integers $k$ such that $G$ admits a faithful $k$-transitive action. Few obstructions are known to impose an upper bound on the transitivity degree for infinite groups. The results of this article provide two new classes of groups whose transitivity degree can be computed, as a corollary of a classification of all $3$-transitive actions of these groups. More precisely, suppose that $G$ is a subgroup of the homeomorphism group of the circle $\homeo(\Sbb^1)$ or the automorphism group of a tree $\aut(\T)$. Under natural assumptions on the  stabilizers of the action of $G$ on $\Sbb^1$ or $\partial T$, we use the dynamics of this action to show that every faithful action of $G$ on a set that is at least $3$-transitive must be conjugate to the  action of $G$ on one of its orbits in $\Sbb^1$ or $\partial \T$.
\end{abstract}


\setcounter{tocdepth}{2}


\section{Introduction}

An infinite permutation group $G \leq \Sym(\Omega)$ acting on a set $\Omega$ is \textbf{highly transitive} if $G$ is $k$-transitive for all $k$. Recall that \textbf{$k$-transitive} means that $G$ acts transitively on the set of ordered $k$-tuples of distinct elements of $\Omega$. If $\Sym(\Omega)$ is endowed with the pointwise convergence topology, $G$ is highly transitive if and only if $G$ is a dense subgroup of $\Sym(\Omega)$. We say that an abstract group is highly transitive if it admits a faithful action on a set $\Omega$ that is highly transitive; equivalently if it admits a dense embedding into $\Sym(\Omega)$.

The class of countable highly transitive groups is not as restricted as one could expect at first sight. Hull and Osin  showed that every countable acylindrically hyperbolic group admits a highly transitive action with finite kernel \cite{Hull-Osin}. This generalized a long list of previous results, notably about finitely generated free groups \cite{McDonough-free,Dixon-free,Olshanskii-free}, surface groups \cite{Kitroser-surface}, hyperbolic groups \cite{Chaynikov}, free products \cite{Hickin-freeprod,Gunhouse-freeprod,Moon-Stalder} or outer automorphism groups of free groups \cite{Garion-Glasner}. In a similar direction the existence of a sufficiently rich geometric or dynamical action was also exploited by Gelander and Glasner in connection with the problem of the existence of faithful primitive actions \cite{Gel-Glas-prim}.

A quite different source of examples of highly transitive groups  are topological full groups of minimal \'etale groupoids over the Cantor set. This class of groups give rise to highly transitive groups that are finitely generated and simple  \cite{Mat-SFT, Nek-simple}, including some which are amenable \cite{Ju-Mo} or have intermediate growth \cite{Nek-frag}, as well as non-amenable ones (such as Thompson's group $V$).

The \textbf{transitivity degree} $td(G)$ of an abstract group $G$ is the supremum of all $k$ such that $G$ admits a faithful $k$-transitive action. The transitivity degree is always at least $1$, and is clearly an invariant of the group. The study of multiply transitive finite permutation groups has a long history, a treatment of which can be consulted in \cite{Wielandt-book,Cameron-book,Dixon-Mortimer}. The classification of the finite simple groups led to a classification of the multiply transitive finite permutation groups, so that transitivity degrees of finite groups are completely determined (see for instance \cite{Cameron-book,Dixon-Mortimer}).

The case of infinite groups is quite different\footnote[2]{Although not directly related to the problems considered here, let us also mention that the study of sharply $k$-transitive groups is also very different for finite and infinite groups (recall that $G \leq \Sym(\Omega)$ is sharply $k$-transitive if $G$ acts freely transitively on the set of ordered $k$-tuples of distinct elements of $\Omega$). We refer to \cite[Vol. I]{Tits-oeuvres} and also to \cite{Glas-Gulk-lin,RST-2trans,Tent-3trans,Rips-Tent-2trans} for more recent developments on this topic.}. Highly transitive groups clearly have infinite transitivity degree. In general there are few obstructions that are known to impose an upper bound on the transitivity degree, and these are mostly of algebraic nature. One obstruction is the existence of non-trivial commuting normal subgroups $M,N \lhd G$. This includes for instance groups having a non-trivial abelian normal subgroup. The relevance of this property in this setting comes from the following classical result: if $G$ admits a primitive faithful action on a set $\Omega$, the actions of $M,N$ on $\Omega$ must be conjugate to  the left and right regular representation of the same group, and $M,N$ are the only minimal normal subgroups of $G$ (see e.g.\ \cite[Theorem 4.4]{Cameron-book}). This implies that $G$ has transitivity degree at most $3$ \cite[Theorem 7.2.A]{Dixon-Mortimer}, and that if $G$ has transitivity degree at least $2$ then $G$ has minimal normal subgroups. So for instance infinite residually finite solvable groups have transitivity degree $1$.  A second obstruction comes from the observation that in a $k$-transitive permutation group, the setwise stabilizer of any $k$-tuple surjects onto $\sym(k)$. It follows that $td(G) < k$ if $G$ does not admit $\sym(k)$ as a subquotient (this is used repeatedly in \cite{Hull-Osin}).  This implies for instance that any torsion group $G$ containing no element of order $k$ has $td(G)<k$, or that any group $G$ satisfying  a non-trivial law verifies $td(G)<\infty$ (since no non-trivial law is satisfied by all finite symmetric groups). Recall that mixed identities in groups are generalizations of laws (see below for definitions). It was proven in \cite{Hull-Osin} that groups satisfying  a non-trivial mixed identity cannot be highly transitive, unless they contain a group of finitary alternating permutations as a normal subgroup. We show in the appendix that actually such groups have bounded transitivity degree, where the bound depends only on the length of the mixed identity (see Proposition \ref{prop-mixid-td}).

For infinite countable groups, the aforementioned classes of groups seem to be the only ones for which an upper bound on the transitivity degree is known. The results of this paper provide two new classes of groups whose transitivity degree can be computed, via an approach of dynamical nature. Both classes are defined in terms of a one-dimensional action, namely an action either on the circle or on a simplicial tree. According to the references mentioned earlier, many groups acting on the circle or on a tree are known to be highly transitive. A crucial point  in our approach is that in both settings we require the  action  to be \emph{non-topologically free} (see below for definitions). For these two classes of groups, we obtain that the transitivity degree is always at most $3$. This bound is derived from a much more precise result, as we actually provide a complete description of all 3-transitive actions of the groups under consideration.

\subsection{Groups acting on the circle}

We denote by $\homeo(\Sbb^1)$ the group of homeomorphisms of the circle, and by $\homeo^+(\Sbb^1)$ the subgroup of index $2$ that preserves the orientation. More generally for a subgroup $G \leq \homeo(\Sbb^1)$, we write $G^+ := G \cap \homeo^+(\Sbb^1)$.

Group actions on the circle can be classified into a finite number of types according to their dynamics; see Section \ref{sec-circle} for background and definitions. The case of minimal and proximal actions is in some sense the \enquote{generic} situation. Many countable subgroups $G\le \homeo(\Sbb^1)$ whose action on $\Sbb^1$ is minimal and proximal are known to admit faithful highly transitive actions by the results mentioned above (for instance finitely generated free subgroups, lattices in $\psl(2, \R)$ acting on $\Sbb^1=\mathbb{P}^1(\R)$ by projective transformations). A common feature of these examples is that their action on $\Sbb^1$ is topologically free. Recall that an action of a group $G$ on a topological space $X$ is \textbf{topologically free} if the set of fixed points of every non-trivial element $g\in G$ has empty interior. When $G$ is a discrete group, $X$ is a compact space and the action of $G$ on $X$ is minimal, this is equivalent to saying that the stabilizer URS \cite{Gla-Wei} associated to the action of $G$ on $X$ is trivial.

We focus instead on groups whose action on $\Sbb^1$ is \emph{not} topologically free. This is equivalent to the existence of  $g \in G$, $g \neq 1$, and an open interval $I \subset \Sbb^1$ such that $g$ fixes $I$ pointwise. This assumption is obviously satisfied by various \enquote{large}  groups such as $\homeo(\Sbb^1)$, $\mathrm{Diff}^k(\Sbb^1)$ or $\mathrm{Diff}^\omega(\Sbb^1)$; but  also by some well-studied countable groups. Examples are Thompson's group $T$ (which is simple and finitely presented) and its many generalizations, or groups of piecewise projective homeomorphisms considered in \cite{Monod-Pw}. 

Our first result says that for many groups $G \leq \homeo(\Sbb^1)$ with a non topologically free action, one can completely describe all $3$-transitive faithful actions of $G$. Two actions of a group $G$ on $\Omega$ and $\Omega'$  are \textbf{conjugate} if there exists a bijective $G$-equivariant map $\Omega \to \Omega'$.

\begin{thm} \label{thm-intro-circle}
Let $G \leq \homeo(\Sbb^1)$ be a group of homeomorphisms of $\Sbb^1$ whose action on $\Sbb^1$ is minimal, proximal, and not topologically free. Assume that distinct points in $\Sbb^1$ have distinct stabilizers in $G^+$. Then for every faithful, 3-transitive action of $G$ on a set $\Omega$,  there exists a $G$-orbit $\mathcal{O}\subset \Sbb^1$ such that the action of $G$ on $\Omega$ is conjugate to the action on $\mathcal{O}$.
\end{thm}

That distinct points of the circle have distinct stabilizers in $G^+$ means that the map on $\Sbb^1$ defined by $z \mapsto G_z^+$ is injective. For non topologically free actions, this assumption does not seem to be very restrictive. It is satisfied in all the examples of groups mentioned earlier (and we actually do not know if it follows automatically from the other assumptions in the theorem).

Recall that no group $G \leq \homeo(\Sbb^1)$ can act $4$-transitively on an orbit in $\Sbb^1$, and if $G$ preserves the orientation then $G$ cannot act $3$-transitively on any of its orbits in $\Sbb^1$. Hence, since any $k$-transitive action with $k \geq 3$ is obviously $3$-transitive, it is an immediate corollary of Theorem \ref{thm-intro-circle} that every group $G$ as in the statement satisfies $td(G)\leq 3$. It turns out that, in the course of the proof, we can actually obtain this result without the assumption on point stabilizers.

\begin{thm} \label{thm-bis-intro-circle}
Let $G \leq \homeo(\Sbb^1)$ be a group of homeomorphisms of $\Sbb^1$ such that the action of $G$ on $\Sbb^1$ is minimal, proximal, and not topologically free. Then the transitivity degree of $G$ is at most $3$; and if $G \leq \homeo^+(\Sbb^1)$, then the transitivity degree of $G$ is at most $2$.
\end{thm}

In several motivating examples, the group $G$ admits an orbit in $\Sbb^1$ on which it acts $2$ or $3$-transitively. In these situations Theorem \ref{thm-bis-intro-circle} allows to determine exactly the transitivity degree of $G$.  This is the case, for instance, for the entire group of homeomorphisms or diffeomorphisms of $\Sbb^1$. An example where $G$ is finitely generated is given by the following:

\begin{cor} \label{cor-intro-S1-T}
The transitivity degree of Thompson's group $T$ is equal to $2$, and the transitivity degree of  Thompson's group with flips  $T^{\pm}$ is equal to $3$.
\end{cor}

We refer to \cite{CFP} for an introduction to the Thompson groups $F,T$ and $V$. Recall the group $V$ is a group of homeomorphisms of the Cantor set, and $V$ acts highly transitively on its orbits in the Cantor set. In the above statement $T^{\pm}$ is the subgroup of $\homeo(\Sbb^1)$ that admits the same description as the group $T$ (see \S \ref{subsec:ThompsonT}), except that $T^{\pm}$ is allowed to reverse the orientation. It contains the group $T$ as a subgroup of index $2$.

In Section \ref{sec-line} we consider the case of groups acting on the real line. We show that if $G \leq \homeo^+(\mathbb{R})$ acts on $\mathbb{R}$ with no fixed points and $G$ contains non-trivial elements with compact support, then the transitivity degree of $G$ is at most $2$ (Proposition \ref{prop-line-td2}). This bound is a priori not as optimal as the one from Theorem \ref{thm-bis-intro-circle}, and we do not know whether this could be improved from $2$ to $1$. An example to which Proposition \ref{prop-line-td2} applies is Thompson's group $F$ (Corollary \ref{cor-line-F}).

\subsection{Groups acting on trees}

The second class of groups under consideration are defined in terms of an action on a tree. In the sequel $\T$ is a simplicial tree, and we denote by $\aut(\T)$ its group of automorphisms, and by $\aut(\T)^+$ the subgroup of $\aut(\T)$ that preserves the natural bipartition of the vertex set of $\T$. The group $\aut(\T)^+$ has index at most $2$ in $\aut(\T)$.

The class of groups acting on trees $G \leq \autT$ is also known to be a source of highly transitive groups; see \cite{Fim-Moon-Stal} and the references mentioned above about free groups and free products. In all these examples the action of $G$ on the boundary of the tree $\partial T$ is topologically free. For free groups and free products this is clear since the stabilizer of an edge in the associated tree is trivial, and for the examples from \cite{Fim-Moon-Stal} this follows from the condition that is imposed there on stabilizers of edges \cite[Def.\ 1.1]{Fim-Moon-Stal}.

As in the case of the circle,  we focus instead on  subgroups $G \leq \autT$ whose action on $\partial \T$ is \emph{not} topologically free. In this setting this is equivalent to  the existence of  a half-tree $A \subset \T$ (see Section \ref{sec-trees} for the terminology) and $g \in G$, $g \neq 1$, such that $g$ acts trivially on $A$. Uncountable groups with this property include the group $\aut(\T_d)$, or the Burger-Mozes universal group $U(F)$ with local action prescribed by a finite permutation group $F \leq \Sym(d)$ acting non-freely on $d$ letters \cite{BM-IHES}. Countable examples, and actually finitely generated, include the groups $G(F,F')$ with almost prescribed local action studied in  \cite{LB-ae}, where $F \lneq F' \leq \Sym(d)$ and the permutation group $F$ acts freely on $d$ letters. Other countable examples can be obtained as piecewise prescribed groups as in \cite[Section 4]{LB-cs}.

The second main result of the article says that in this setting we have the exact same phenomena as in Theorem \ref{thm-intro-circle}:

\begin{thm} \label{thm-intro-tree}
Suppose that the action of $G \leq \autT$ on $\T$ is minimal and of general type, and that the action of $G$ on $\partial \T$ is not topologically free. Assume that $G$ acts faithfully and 3-transitively on a set $\Omega$. Then there exists a $G$-orbit $\mathcal{O}\subset \partial \T$ such that the actions of $G$ on $\Omega$ and on $\mathcal{O}$ are conjugate.
\end{thm}

After the first version of this article appeared, a converse of this theorem was proven in \cite{FLMMS}, namely that if $G\le \aut(\T)$ is minimal and of general type and the action of $G$ on $\partial \T$ \emph{is}  topologically free, then the group $G$ is highly transitive.

Regular trees share with the circle $\Sbb^1$ the property of being a source of examples of actions that are either $3$-transitive but not $4$-transitive, or $2$-transitive but not $3$-transitive. Indeed, it is well-known that for $d \geq 3$ the group $\aut(\T_d)$ acts $3$-transitively on $\partial \T_d$, and that no subgroup $G \leq \aut(\T_d)$ can act $4$-transitively on any of its orbits in $\partial \T_d$. Similarly $\aut(\T_d)^+$ acts $2$-transitively on $\partial \T_d$, and no subgroup $G \leq \aut(\T_d)^+$ acting minimally on $T_d$ can act $3$-transitively on some orbit in $\partial \T_d$. So as in the case of the circle, the following is a direct consequence of the theorem:

\begin{cor} \label{cor-intro-tree}
Let $G \leq \autT$ as in Theorem \ref{thm-intro-tree}. Then the transitivity degree of $G$ is at most $3$. If in addition $\T = \T_d$ is a regular tree and $G \leq \aut(\T_d)^+$, then the transitivity degree of $G$ is at most $2$.
\end{cor}

Corollary \ref{cor-intro-tree} allows to determine the transitivity degree of some of the groups mentioned above, for instance some groups in the families $G(F,F')$ and $G(F,F')^+$. In view of the simplicity results of \cite[Section 4]{LB-ae}, this gives examples of infinite finitely generated simple groups of transitivity degree $2$ (compare with Corollary \ref{cor-intro-S1-T}).

\subsection*{Outline of proofs and organization} 

The proofs of Theorems \ref{thm-intro-tree} and \ref{thm-intro-circle} share a common global structure. Isometric group actions on trees are classified into a finite number of types, according to the existence of hyperbolic elements and of a fixed point or a fixed pair in the boundary (see below). Let $G$ be a group as in Theorem \ref{thm-intro-tree}. Given a  3-transitive action of $G$ on a set $\Omega$ and a finite set of points $F = \left\{ \omega_1,\ldots, \omega_n\right\} \subset \Omega$, we study the type of the action of the stabilizer $G_F$ on the tree. We go through a discussion that considers all possible types for $n=1$ and $n=2$, and we show that the only possibility is that the stabilizer of a point in $\Omega$ fixes a point in $\partial T$. In the case of groups acting on the circle, a similar study is carried out, but the above classification is replaced by the classical trichotomy about minimal invariant subsets, together with the use of a theorem of Margulis \cite{Mar-Tits} and Ghys \cite{Ghys-circ}, which characterizes {proximality} of a group $G\subset\homeo_+(\Sbb^1)$  in terms of its {centralizer} in $\homeo_+(\Sbb^1)$ (see Theorem \ref{t-Margulis}).

In both situations, one stage of the proof makes use of a certain first order formula that is satisfied by $G$, which roughly speaking expresses that certain commutations hold inside $G$. This can be reinterpreted in terms of mixed-identities (see below for the terminology). Although these formulas are not quite the same in the case of the tree and of the circle (compare Lemma \ref{lem-mix-id} and Lemma \ref{lem-commutations-S1}), in both cases they essentially come from the fact that trees and circles are one-dimensional objects. Moreover the information that we extract from them and that is sufficient for our purpose is actually the same in the two situations (see the universal formula (\ref{eq-identity})). Note that the existence of a non-trivial mixed-identity is already enough to ensure that the group has finite transitivity degree (see Proposition \ref{prop-mixid-td}),  but obtaining Theorems 1.1 and 1.4 requires more work.

 Despite having these similarities, the proofs of Theorems \ref{thm-intro-circle} and \ref{thm-intro-tree} locally require quite different arguments, so that we have decided to give them separately instead of trying to provide a simultaneous treatment that would have made the discussion more obscure. The case of groups acting on trees is slightly simpler, so that we do it first (Section \ref{sec-trees}). The case of groups acting on the circle is carried out in Section \ref{sec-circle}, and finally Section \ref{sec-line} deals with groups acting on the real line. 

\subsection*{Acknowledgments} The authors are very grateful to Yair Glasner, whose talk in Bernoulli Center on related problems was a source of inspiration for this work.

\section{Preliminaries} \label{sec-terminology}

If $\Omega$ is a set, we will denote by $\Omega^{\{ 2\}}$ the set of unordered pairs in $\Omega$. For a permutation group $G \leq \Sym(\Omega)$ and $ \omega \in \Omega$, we will denote by $G_\omega$ the stabilizer of $\omega$ in $G$. For $\Delta \in \Omega^{\{ 2\}}$, $G_\Delta$ will be the pointwise stabilizer of $\Delta$, and $G_{(\Delta)}$ will be the setwise stabilizer of $\Delta$. Observe that when $G$ is transitive on $\Omega$, all the subgroups $G_\omega$ are conjugated in $G$. Similarly when $G$ is $2$-transitive, all the subgroups $G_\Delta$ are conjugated. In particular any property that is invariant by conjugation and that holds for one of these subgroups automatically holds for all of them. This observation will be used implicitly throughout Sections \ref{sec-trees} and \ref{sec-circle}. 

Recall that a permutation group $G \leq \Sym(\Omega)$ is \textbf{primitive} if there is no $G$-invariant partition of $\Omega$ apart from the trivial ones. A transitive permutation group is primitive if and only if point stabilizers are maximal subgroups of $G$. Any $2$-transitive permutation group is primitive. If $N$ is a non-trivial normal subgroup of $G$ and $G$ is primitive, then $N$ is transitive on $\Omega$. 

When $g,h$ are elements of a group $G$, we use the notation $h^g = ghg^{-1}$ for the conjugate of $h$ by $g$, and we use the convention $[g,h] = ghg^{-1}h^{-1}$ for the commutator of $g$ and $h$.

\section{Groups acting on trees} \label{sec-trees}

We recall some some background about group actions trees. Proofs can be found in \cite{Pa-Va,Tits}. An automorphism $g$ of a simplicial tree $\T$  is \textbf{elliptic} if $g$ stabilizes a vertex or an edge, and $g$ is \textbf{hyperbolic} if there exists a bi-infinite geodesic line, called the axis of $g$, along which $g$ acts as a non-trivial translation. In the latter situation we also say that $g$ is a translation. The two ends defined by the axis are called the endpoints of $g$. Any automorphism of $\T$ is either elliptic or hyperbolic. For a group $G$ acting on $\T$, the $G$-action on $\T$ is \textbf{minimal} if there is no proper $G$-invariant subtree. Any group $G$ containing a translation admits a unique minimal invariant subtree, which is the union of the axis of translations of $G$.

If $e$ is an edge and $v$ a vertex of $e$, the subtree of $\T$ spanned by vertices whose projections to $e$ is equal to $v$ is called a \textbf{half-tree}. Note that if $G$ acts minimally on $\T$ and if $\T$ is not a single vertex or a single edge, then every half-tree of $\T$ is infinite.

Group actions on trees enjoy the following classification:

\begin{prop}
Every group $G \leq \autT$ satisfies exactly one the following:
\begin{enumerate}
	\item $G$ stabilizes a vertex or an edge.
	\item $G$ is \textbf{horocyclic}: $G$ contains no translation and $G$ fixes a unique point in $\partial \T$.
	\item $G$ is \textbf{lineal}: $G$ contains translations, and all translations in $G$ share the same axis.
	\item $G$ is \textbf{focal}: $G$ contain translations and $G$ fixes a unique point in $\partial \T$.
	\item $G$ is of \textbf{general type}: there exist two translations of $G$ not sharing any endpoint. 
\end{enumerate}
\end{prop}

In the sequel we will repeatedly use the following easy fact:

\begin{lem}
Let $G \leq \autT$ be a subgroup such that the action of $G$ on $\T$ is minimal and of general type. Then given two half-trees $A, B \subset \T$, there exists $g \in G$ such that $g(A) \subseteq B$.
\end{lem}

\begin{proof}
Without loss of generality we may assume that $A,B$ are disjoint. In this case, any translation $g$ whose axis is contained in $B$ satisfies $g(A) \subseteq B$. 
\end{proof}

\subsection{Preliminaries}

The goal of this section is to establish preliminary results towards the proof of Theorem \ref{thm-intro-tree}.

\begin{lem} \label{lem-class-ntrans-type}
Let $G \leq \autT$ whose action on $\T$ is minimal and of general type. Suppose that $\Omega$ is a set on which $G$ acts faithfully, transitively, and with finitely many orbits on $\Omega \times \Omega$. Then for $\omega \in \Omega$, the subgroup $G_\omega$ must act minimally on $\T$, and $G_\omega$ is either focal or of general type.
\end{lem}

\begin{proof}
We first show that $G_\omega$ acts minimally on $\T$. We argue by contradiction and assume that $G_\omega$ preserves a proper invariant subtree $A \subseteq T$. Recall that for a transitive permutation group, the number of $G$-orbits in the product $\Omega \times \Omega$ is equal to the number of double cosets $G_\omega \backslash G / G_\omega$. By assumption this quantity is finite, and hence there exist $g_1,\ldots,g_r \in G$ such that $G = \cup G_\omega g_i G_\omega$. Given an element $g \in G$, we consider $d(gA,A) = \inf_{v,w \in A} d(gv,w)$. By assumption there are $\gamma, \gamma' \in G_\omega$ and some $g_i$ such that $g = \gamma g_i \gamma'$, so that $d(gA,A) = d(g_iA,A)$ because $\gamma$ and $\gamma'$ preserve $A$ by assumption. In particular we have $\sup_g d(gA,A) < \infty$. Now for an arbitrary $R > 0$, since $G$ acts minimally on $\T$ we may always find a half-tree $B$ in $\T$ such that $d(A,B) > R$, and some hyperbolic element $g$ of $G$ whose axis is contained inside $B$. For such an element $g$ we have that $g(A)$ is included inside $B$, and in particular $d(gA,A) \geq d(B,A) > R$. Since $R$ is arbitrary, we have obtained our contradiction.

We now shall prove that $G_\omega$ is either focal or of general type. By the previous paragraph it suffices to argue that $G_\omega$ cannot be horocyclic. Again we assume for a contradiction that $G_\omega$ is horocyclic. Let us fix some hyperbolic element $g$ in $G$. Since the double coset space $G_\omega \backslash G / G_\omega$ is finite, we may find $m > n \geq 1$ such that $g^m$ and $g^n$ fall into the same double coset. Let $h, h' \in G_\omega$ such that $g^m = h g^n h'$, and choose a vertex $x$ in $\T$ sufficiently close to the fixed point of $G_\omega$ so that $h,h'$ both fix $x$. Then we have $d(g^mx,x) = d(h g^n h'x,x) = d(g^n x,x)$ because $hx=h'x=x$. Now recall that for a hyperbolic isometry $\gamma$ of a tree and for a vertex $v$, we have the formula $d(\gamma v, v) = 2 d(v,L_\gamma) + L(\gamma)$, where $L_\gamma$ is the axis of $\gamma$ and $L(\gamma)$ is the translation length of $\gamma$ (see \cite[Prop. 3.2]{Tits}). Here since the axis of $g^m$ and $g^n$ is the axis of $g$, it follows that we have $2 d(x,L_g) + L(g^m) = 2 d(x,L_g) + L(g^n)$, i.e.\ $m L(g) = L(g^m)  = L(g^n) = nL(g)$. This is absurd because $L(g) > 0$ and $m > n$ by assumption.
\end{proof}

Given a half-tree $A \subseteq T$ and $G \leq \autT$, we denote by $G_A$ the subgroup of $G$ consisting of elements that fix pointwise the complement of $A$. In other words, $G_A$ consists of the elements of $G$ that are supported inside $A$.

\begin{lem} \label{lem-stab-orbits-size3}
Let $G \leq \autT$ be minimal and of general type, and such that the action of $G$ on $\partial \T$ is not topologically free. Suppose that $G$ acts $2$-transitively on a set $\Omega$ in such a way that  for every $\omega \in \Omega$, the subgroup $G_\omega$ is of general type. Then for every half-tree $A \subseteq T$, all orbits of $G_A$ in $\Omega$ have cardinality at least $3$.
\end{lem}

\begin{proof}
Assume by contradiction that $G_A$ has an orbit of cardinality at most $2$. Then its commutator subgroup $[G_A,G_A]$ must fix a point $\omega \in \Omega$. (Note that the group $G_A$ is always non-abelian: in fact given any non-trivial  $f\in G_A$ , choose a a half-tree $B\subset A$ such that $f(B) \cap B=\varnothing$, then $fG_Bf^{-1}=G_{f(B)}$ and thus $f$ cannot commute with any element of $G_B$).  Since $G_\omega$ is minimal by Lemma \ref{lem-class-ntrans-type} and $G_\omega$ is of general type by assumption, it follows that if $B$ is an arbitrary half-tree in $\T$, then there is $g \in G_\omega$ such that $g(B) \subseteq A$. Therefore $gG_Bg^{-1} \leq G_A$, and it follows that $G_\omega$ contains $[G_B,G_B]$. All these subgroups $[G_B,G_B]$ generate a non-trivial normal subgroup of $N$ of $G$ when $B$ varies. Therefore $G_\omega$ contains a non-trivial normal subgroup of $G$, and hence is transitive on $\Omega$: this is absurd.
\end{proof}

\begin{lem} \label{lem-tree-moves}
Suppose we are given three disjoint half-trees $A_1,A_2,A_3 \subset \T$, and let $g \in \autT$. Then either there exists $i$ such that $A_i$ and $g(A_i)$ are disjoint, or $g(A_i) = A_i$ for every $i$.
\end{lem}

\begin{proof}
We denote by $Y$ the tripod defined by the disjoint half-trees $A_1,A_2,A_3$. It is the finite tripod having a center $c$ and extremities $v_1,v_2,v_3$ so that $A_i$ is exactly the set of points $x$ such that the geodesic $[c,x]$ contains $[c,v_i]$, $i=1,2,3$. Suppose that $g$ is a translation. Then at least one $A_i$ contains no endpoint of $g$, and it follows that $g(A_i)$ is disjoint from $A_i$. Suppose that $g$ is elliptic. If $g$ is an inversion around an edge $e$, then $A_i$ and $g(A_i)$ are disjoint whenever $A_i$ does not contain $e$, and this happens for at least two of the $A_i$. When the set of vertices fixed by $g$ is non-empty, then either there is $i$ such that $g$ does not fix $v_i$, and again $A_i$ and $g(A_i)$ are disjoint; or $g$ fixes pointwise $Y$, and it follows that $g$ stabilizes $A_i$ for all $i$.
\end{proof}

\begin{lem} \label{lem-mix-id}
Let $G \leq \autT$, and $g_1,g_2,g_3 \in G$ supported in disjoint half-trees. For $g \in G$, at least one of the following hold:
\begin{enumerate}
	\item \label{item-one-commut} $[g_1^g,g_1]=1 \vee [g_2^g,g_2]=1 \vee [g_3^g,g_3]=1$;
	\item \label{item-all-commut} $[g_1^g,g_2]= [g_1^g,g_3]=  [g_2^g,g_1]=  [g_2^g,g_3]=  [g_3^g,g_1]=  [g_3^g,g_2]=1$.
\end{enumerate}

In particular $G$ satisfies \begin{equation} \label{eq-identity} \forall g \, \, [g_1^g, g_1] = 1 \vee [g_1^g, g_2] = 1 \vee \ldots \vee [g_3^g, g_3] = 1. \tag{$\star$} \end{equation}
\end{lem}


\begin{proof}
Let $A_1,A_2,A_3$ be disjoint half-trees such that $g_i$ is supported in $A_i$. Suppose that there is $i$ such that $A_i$ and $g(A_i)$ are disjoint. Then the elements $g_i^g$ and $g_i$ have disjoint support, and hence commute. So we are in situation (\ref{item-one-commut}). Otherwise then by Lemma \ref{lem-tree-moves} we have that $g$ preserves $A_i$ for all $i$. Therefore the conjugate $g_i^g$ has support inside $A_i$, and hence commutes with $g_j$ for $j \neq i$, and we are in situation (\ref{item-all-commut}). The formula (\ref{eq-identity}) immediately follows from the first statement.
\end{proof}

In the sequel we will use Lemma \ref{lem-mix-id} through (\ref{eq-identity}), but for completeness we note that it has the following reinterpretation. Recall that for $w \in G \ast \mathbb{Z}$, we say that $G$ satisfies the \textbf{mixed-identity} $w=1$ if every homomorphism from $G \ast \mathbb{Z}$ to $G$ that is the indentity on $G$ is trivial on $w$. We say that $G$ satisfies a non-trivial mixed-identity if $w$ is a non-trivial element of $G \ast \mathbb{Z}$. Note that taking $w$ in $G \ast \mathbb{F}_n$ for a free group $\mathbb{F}_n$ of rank $n \geq 2$ would yield an equivalent definition \cite[Remark 5.1]{Hull-Osin}.

\begin{prop} \label{prop--tree-mixed-id}
Suppose that the action of $G \leq \autT$ on $\T$ is minimal and of general type, and that the action of $G$ on $\partial \T$ is not topologically free. Then $G$ satisfies a non-trivial mixed-identity.
\end{prop}

\begin{proof}
The assumptions imply that we can choose non-trivial elements $g_1,g_2,g_3$ of $G$ that are supported in disjoint half-trees. We denote by \[ c_{i,j} = [g_i ^t, g_j] \in G \ast \langle  t \rangle, \] and we let $w$ be the iterated commutator of the nine elements $c_{1,1},\ldots,c_{3,3}$: \[ w = [c_{1,1},[c_{1,2},[ \ldots [c_{3,2}, c_{3,3}]] \ldots ] \] Since $g_1,g_2,g_3$ are non-trivial, the normal form theorem implies that $w$ is a non-trivial element of $G \ast \langle t  \rangle$. That $G$ satisfies the mixed-identity $w = 1$ is a direct consequence of the formula (\ref{eq-identity}) (Lemma  \ref{lem-mix-id}).
\end{proof}





\subsection{The proof of Theorem \ref{thm-intro-tree}}

In all this section we let $G \leq \autT$ be a subgroup that is minimal and of general type, and such that the action of $G$ on $\partial \T$ is not topologically free.

\begin{lem} \label{lem-trees-GDelta-gen}
There does not exist a set $\Omega$ on which $G$ acts faithfully and $3$-transitively and such that for $\Delta \in \Omega^{\{ 2\}}$, the action of the subgroup $G_\Delta$ on $\T$ is of general type.
\end{lem}

\begin{proof}
Argue by contradiction and assume that $\Omega$ is a set on which $G$ acts faithfully and $3$-transitively, and such that $G_\Delta$ is of general type. For $\omega \in \Omega$, observe that the action of $G_\omega$ on $\T$ is minimal by Lemma \ref{lem-class-ntrans-type}. Since the action of $G_\omega$ on the remaining points in $\Omega$ is $2$-transitive, we may now apply Lemma \ref{lem-class-ntrans-type} to the group $G_\omega$, and we deduce that for $\Delta \in \Omega^{\{ 2\}}$ the action of $G_\Delta$ on $\T$ is minimal.

Since the $G$-action on $\partial \T$ is not topologically free, we may find a non-trivial element $g_1 \in G$ supported inside a half-tree $A$ of $\T$. It is easy to see that the subgroup $G_{A}$ cannot be abelian, and hence cannot have all its elements of order two. Hence we may assume that $g_1$ does not have order two, so that $g_1$ admits an orbit of size at least three in $\Omega$. Let $\omega_1,\omega_2,\omega_3$ be distinct points such that $\omega_2 = g_1(\omega_1)$ and $\omega_3 = g_1 (\omega_2)$, and write $\Delta = \left\{\omega_1,\omega_2\right\}$. Since the action of $G_\Delta$ is minimal and of general type, we may find elements $u_2,u_3 \in G_\Delta$ such that if we let $g_2 := u_2 g_1 u_2^{-1}$ and $g_3 := u_3 g_1 u_3^{-1}$, then $g_1,g_2,g_3$ are supported in disjoint half-trees. Observe that $g_2 (\omega_1) = \omega_2$ because $u_2$ fixes both $\omega_1$ and $\omega_2$, and since $g_2$ and $g_1$ commute (because they act on $\T$ with disjoint supports), we also have $g_2(\omega_2) = g_1(\omega_2) = \omega_3$. For the same reason $g_3 (\omega_1) = \omega_2$ and $g_3(\omega_2) = \omega_3$.

Now let $g \in G$ such that $g(\omega_1) = \omega_1$ and $g(\omega_2) = \omega_2$. According to Lemma \ref{lem-mix-id}, there are $i,j$ such that $[g_i^g,g_j]=1$. The element $g_i^g$ sends $\omega_1$ to $\omega_2$, and commutes with $g_j$, so again $g_i^g$ sends $\omega_2$ to $\omega_3$. But $g_i^g(\omega_2) = g g_i (\omega_2) = g(\omega_3)$, so we must have $\omega_3 = g(\omega_3)$. We have thus shown that the fixator of two points $\omega_1,\omega_2$ in $G$ fixes a third point, which is clearly a contradiction with the assumption that the action is $3$-transitive.
\end{proof}

\begin{prop} \label{prop-trees-GDelta-focal}
There does not exist a set $\Omega$ on which $G$ acts faithfully and $3$-transitively and such that for $\Delta \in \Omega^{\{ 2\}}$, the action of the subgroup $G_\Delta$ on $\T$ is focal.
\end{prop}

\begin{proof}
Again we argue by contradiction. Since the subgroup $G_{(\Delta)}$ contains $G_{\Delta}$ as a subgroup of index $2$, the action of $G_{(\Delta)}$ on $\T$ is also focal. We denote by $\xi(\Delta) \in \partial \T$ the unique end of $\T$ that is fixed by $G_{(\Delta)}$. Note that the map \[ \pi: \Omega^{\{ 2\}} \to \partial \T, \, \, \Delta \mapsto \xi(\Delta), \] is $G$-equivariant. For simplicity we will denote by $S_{\Delta}$ the stabilizer of $\xi(\Delta)$ in $G$. By definition $G_{(\Delta)} \leq S_{\Delta}$.

The proof of the proposition will consist in a series of lemmas that will lead us to a contradiction.

\begin{lem} \label{lem-tree-Strans}
Let $\Delta \in \Omega^{\{ 2\}}$. We have $G_{(\Delta)} \neq S_{\Delta}$, and $S_{\Delta}$ acts transitively on $\Omega$.
\end{lem}

\begin{proof}
Suppose that $S_{\Delta}$ is equal to $G_{(\Delta)}$. For every half-tree $A$ not containing the point $\xi(\Delta)$, the subgroup $G_A$ fixes $\xi(\Delta)$, and hence $G_A \leq G_{(\Delta)}$. So $G_A$ has an orbit in $\Omega$ of size at most $2$, which contradicts Lemma \ref{lem-stab-orbits-size3}. So $G_{(\Delta)} \neq S_{\Delta}$. The fact that $S_{\Delta}$ acts transitively on $\Omega$ follows immediately because $G_{(\Delta)}$ acts transitively on the complement of $\Delta$ in $\Omega$ by $3$-transitivity of the $G$-action on $\Omega$.
\end{proof}

\begin{lem} \label{lem-tree-SDelta-2tran}
For every $\Delta \in \Omega^{\{ 2\}}$, the fiber $\pi^{-1}(\xi(\Delta))$ forms a partition $\mathcal{P}_\Delta$ of $\Omega$ into blocks of size $2$, and $S_{\Delta}$ preserves $\mathcal{P}_\Delta$ and acts $2$-transitively on its blocks.
\end{lem}

\begin{proof}
Fix $\Delta \in \Omega^{\{ 2\}}$. We first argue that two distinct $\Delta', \Delta'' \in \pi^{-1}(\xi(\Delta))$ are always disjoint. Assume this is not the case. Then $\Delta' \cap \Delta''$ is a singleton, say $\Delta' \cap \Delta'' = \left\{w\right\}$. The action of $G_w$ on $\Omega \setminus \left\{w\right\}$ being $2$-transitive, the subgroup generated by $G_{\Delta'}$ and $G_{\Delta''}$ is equal to $G_w$. Since $G_{\Delta'}$ and $G_{\Delta''}$ both fix $\xi(\Delta)$ by definition of $\Delta', \Delta''$, it follows that the subgroup $G_w$ also fixes $\xi(\Delta)$. This contradicts our assumption that the action of $G_w$ on $\T$ is of general type. 

In order to see that $\pi^{-1}(\xi(\Delta))$ indeed defines a partition of $\Omega$, it is therefore enough to see that the pairs in this fiber cover $\Omega$. The union $\Sigma$ of these pairs defines a subset of $\Omega$ that is $S_{\Delta}$-invariant by equivariance of the map $\pi$. But the subgroup $S_{\Delta}$ is transitive on $\Omega$ by Lemma \ref{lem-tree-Strans}, so we must have  $\Sigma = \Omega$, as desired.

That the subgroup $S_{\Delta}$ preserves $\mathcal{P}_\Delta$ is clear. The action of $S_{\Delta}$ on the blocks of $\mathcal{P}_\Delta$ is transitive because the action of $S_{\Delta}$ on $\Omega$ already is, so in order to show that this action is $2$-transitive we need to show that $G_{(\Delta)}$ acts transitively on $\mathcal{P}_\Delta \setminus \left\{ \Delta\right\}$. Let $\Delta', \Delta'' \in \mathcal{P}_\Delta$ which are distinct from $\Delta$. Choose $\omega' \in \Delta'$ and $\omega'' \in \Delta''$. Since $\Delta', \Delta''$ are disjoint from $\Delta$ we have $\omega', \omega'' \in \Omega \setminus \left\{ \Delta\right\}$. Since $G_{\Delta}$ acts transitively on $\Omega \setminus \Delta$, there exists $g \in G_{\Delta}$ such that $g(\omega') = \omega''$. Since $G_{\Delta}$ preserves the blocks, it follows that we must have $g(\Delta') = \Delta''$, and hence $G_{\Delta}$ acts transitively on $\mathcal{P}_\Delta \setminus \left\{ \Delta\right\}$. A fortiori the same is true for $G_{(\Delta)}$.
\end{proof}

In the sequel we will denote by $S_{\Delta}^0$ the subgroup of $S_{\Delta}$ consisting of elements having trivial germs around the point $\xi(\Delta)$, i.e.\ elements $g \in S_{\Delta}$ such that there exists a half-tree $A \subset \T$ such that $\xi(\Delta) \in \partial A$ and $g$ acts trivially on $A$. Observe that $S_{\Delta}^0$ is a normal subgroup of $S_{\Delta}$, which is non-trivial by the assumption that the $G$-action on $\partial \T$ is not topologically free.

\begin{lem} \label{lem-tree-S-germs}
For every $\Delta \in \Omega^{\{ 2\}}$, we have $S_{\Delta}^0 G_{\Delta} = S_{\Delta}$.
\end{lem}

\begin{proof}
According to Lemma \ref{lem-tree-SDelta-2tran}, the action of $S_{\Delta}$ on the blocks of the partition $\mathcal{P}_\Delta$ is $2$-transitive. In particular this action is primitive. Since $S_{\Delta}^0$ is a normal subgroup of $S_{\Delta}$, it follows that the action of $S_{\Delta}^0$ on these blocks is either trivial or transitive. If this action was trivial then $S_{\Delta}^0$ would be an elementary abelian $2$-group since all blocks have size $2$, which is absurd. So this action is transitive, and it follows that $S_{\Delta}^0$ has at most $2$ orbits in $\Omega$ and each orbit intersects each block of $\mathcal{P}_\Delta$. We claim that this implies that the subgroup $S_{\Delta}^0 G_{\Delta}$ acts transitively on $\Omega$ (note that $S_{\Delta}^0 G_{\Delta}$ is indeed a subgroup because $G_{\Delta}$ normalizes $S_{\Delta}^0$). If $S_{\Delta}^0$ has only one orbit then this is clear. If $S_{\Delta}^0$ has two orbits, then the subgroup $G_{\Delta}$ clearly does not preserve the partition of $\Omega$ into these two orbits because $G_{\Delta}$ acts transitively on $\Omega \setminus \Delta$, so it follows that in any case $S_{\Delta}^0 G_{\Delta}$ has only one orbit. Since for $w \in \Delta$ we have equality $G_{w} \cap S_{\Delta} = G_{\Delta}$, it follows that $S_{\Delta}^0 G_{\Delta}$ is a subgroup of $S_{\Delta}$ that is transitive on $\Omega$ and that contains the stabilizer of a point, so finally $S_{\Delta}^0 G_{\Delta} = S_{\Delta}$.
\end{proof}

\begin{lem} \label{lem-tree-loc-exch}
Fix $\Delta \in \Omega^{\{ 2\}}$. Then for every half-tree $A \subset \T$ there exists $g \in G_A$ such $g$ exchanges the two elements of $\Delta$.
\end{lem}

\begin{proof}
Choose an element $g \in G$ such that $g$ exchanges the two elements of $\Delta$. According to Lemma \ref{lem-tree-S-germs}, there exist $h \in S_{\Delta}^0$ and $h' \in G_{\Delta}$ such that $g = h h'$. So the element $h$ also exchanges the two elements of $\Delta$, and since $h \in S_{\Delta}^0$ there exists a half-tree $B$  such that $h$ is trivial outside $B$ and $\xi(\Delta) \notin \partial B$.

Now let $A$ be an arbitrary half-tree. Without loss of generality we may assume that $A$ is disjoint from $B$ and  $\xi(\Delta) \notin \partial A$. Since the action of $G_{\Delta}$ on $\T$ is minimal and focal, it is possible to find a translation $\gamma \in G_{\Delta}$ such that the axis $L_\gamma$ of $\gamma$ intersects $A$ along an infinite geodesic ray, and $L_\gamma \cap B = \emptyset$. It follows that, upon passing to a suitable power of $\gamma$, we may assume that $\gamma(B) \subset A$, so that the element $\gamma h \gamma^{-1}$ is trivial outside $A$. Note that this element also exchanges the two elements of $\Delta$ since $\gamma \in G_{\Delta}$, so we are done.
\end{proof}

\begin{lem} \label{lem-tree-loc-3cyc}
Fix $\Delta \in \Omega^{\{ 2\}}$ and $\Delta' \in \mathcal{P}_\Delta$ distinct from $\Delta$. Then there exists a half-tree $A$ such that $\xi(\Delta) \notin \partial A$ and $s \in G_A$ such that $s(\Delta) = \Delta'$ and $s(\Delta') \neq \Delta$.
\end{lem}

\begin{proof}
Since the group $S_{\Delta}^0$ is not an extension of two elementary abelian $2$-groups, there exists an element of $S_{\Delta}^0 $ having a cycle of length at least $3$ for the action of $S_{\Delta}^0$ on $\mathcal{P}_\Delta$, i.e.\ there exist an element $s_0$ and distinct $\Delta_1, \Delta_2, \Delta_3 \in \mathcal{P}_\Delta$  such that $s_0(\Delta_i) = \Delta_{i+1}$, $i=1,2$. Note that $s_0$ is indeed supported inside a half-tree that does not contain $\xi(\Delta)$. Now since the $S_{\Delta}$-action on $\mathcal{P}_\Delta$ is $2$-transitive by Lemma \ref{lem-tree-SDelta-2tran}, we may find $g \in S_{\Delta}$ such that $g(\Delta_1) = \Delta$ and $g(\Delta_2) = \Delta'$, and it follows that $s := g s_0 g^{-1}$ satisfies the conclusion.
\end{proof}

We shall now terminate the proof of the proposition. Let $\Delta, \Delta'$ and $s$ be as in Lemma \ref{lem-tree-loc-3cyc}, and let $\Lambda \in \Omega^{\{ 2\}}$ such that $ \Lambda \cap \Delta$ and $ \Lambda \cap \Delta'$ are non-empty. Choose a half-tree $B$ that is disjoint from $A$ and such that $\partial B$ contains neither $\xi(\Delta)$ nor $\xi(\Lambda)$. We apply Lemma \ref{lem-tree-loc-exch} to the point $\xi(\Lambda)$ and find $t$ supported in $B$ such that $t$ exchanges the two elements of $\Lambda$. By construction $t$ fixes $\xi(\Delta)$, and hence $t$ must preserve the partition $\mathcal{P}_\Delta$. Since $t$ exchanges two elements of $\Delta$ and $\Delta'$, it follows that $t$ actually exchanges the blocks $\Delta$ and $\Delta'$. But $t$ and $s$ are supported in disjoint half-trees, and hence commute, and we have a contradiction with $s(\Delta') \neq \Delta$.
\end{proof}

\begin{proof}[Proof of Theorem \ref{thm-intro-tree}]
Assume that $\Omega$ is a set on which $G$ acts faithfully and 3-transitively.  According to Lemma \ref{lem-class-ntrans-type}, for $w \in \Omega$, the action of $G_w$ on $\T$ is minimal, and is either focal or of general type. Assume that $G_w$ is of general type. Since $G_w$ acts $2$-transitively on $\Omega \setminus \left\{w\right\}$, applying Lemma \ref{lem-class-ntrans-type} again we see that for $\Delta \in \Omega^{\{ 2\}}$, the action of $G_\Delta$ on $\T$ is focal or of general type. But we have shown that none of these are possible, respectively in Lemma \ref{lem-trees-GDelta-gen} and in Proposition \ref{prop-trees-GDelta-focal}. Therefore we have reached a contradiction, and it follows that the action of $G_w$ on $\T$ must be focal. If $\xi(w)$ is the unique point of $\partial \T$ that is fixed by $G_w$, then the map $\pi: \Omega \to \partial \T$, $w \mapsto \xi(w)$, is injective and $G$-equivariant, and hence the $G$-action on $\Omega$ is conjugate to the action on $\mathcal{O} = \pi(\Omega)$. This terminates the proof.
\end{proof}

\begin{proof}[Proof of Corollary \ref{cor-intro-tree}]
Suppose that the action of $G \leq \autT$ on $\T$ is minimal and of general type, and that the action of $G$ on $\partial \T$ is not topologically free. Assume that there exists a set $\Omega$ on which $G$ acts faithfully and $4$-transitively. Then this action is in particular $3$-transitive, so it follows from Theorem \ref{thm-intro-tree} that the action of $G$ on $\Omega$ is conjugate to the action of $G$ on one of its orbits $\mathcal{O}$ in $\partial \T$. But then the stabilizer of three distinct points in $\mathcal{O}$ must fix a vertex $v$ of $\T$ (the center of the tripod defined by the three ends), and hence preserves any visual metric in $\partial \T$ associated to $v$. This clearly implies that this subgroup cannot act transitively on the remaining points of $\mathcal{O}$, which is a contradiction with $4$-transitivity.

Assume now that $\T$ is a regular tree and $G \leq \aut(\T_d)^+$. As before any faithful $3$-transitive action of $G$ must be conjugate to the action on one orbit $\mathcal{O}$ in $\partial \T$. Now the stabilizer of two distinct points of $\mathcal{O}$ preserves a bi-infinite geodesic line $L$, and all its elements have even translation length since $G \leq \aut(\T_d)^+$. Hence it follows that if $\xi,\xi'$ are two distinct points of $\mathcal{O}$ whose projections on $L$ are adjacent vertices (such points exist by the assumptions that the tree is regular and the action is minimal), then no element of $G$ can send $\xi$ to $\xi'$ while stabilizing $L$. So the $G$-action on $\mathcal{O}$ is not $3$-transitive.
\end{proof}

\section{Groups acting on the circle} \label{sec-circle}

\subsection{Preliminaries on group actions on the circle} \label{s-circle-prel}

We let $\homeo(\Sbb^1)$ be the group of homeomorphisms of the circle, and $\homeo^+(\Sbb^1)$ its subgroup of index 2 consisting of orientation preserving ones. Similarly for group $G \leq \homeo(\Sbb^1)$, we write $G^+ := G \cap \homeo^+(\Sbb^1)$. Recall that for $G\leq \homeo(\Sbb^1)$, the following trichotomy holds: either $G$ has a finite orbit in $\Sbb^1$; or $G$ acts minimally of $\mathbb{S}^1$, or there exists a unique closed non-empty minimal $G$-invariant subset $K\subset \Sbb^1$, homeomorphic to a Cantor set (see e.g. \cite[Prop. 5.6]{Ghys-circ}). In the latter situation the set $K$ is called an \textbf{exceptional minimal set}. In the third case, the action is \textbf{semi-conjugate} to a minimal action, meaning that there exists a homomorphism $\varphi\colon G\to \homeo(\Sbb^1)$ whose image acts minimally, and a continuous, surjective degree 1 map $h\colon \Sbb^1\to \Sbb^1$ such that $h\circ g=\varphi(g)\circ h$ for every $g\in G$. Note that the type of $G$ (finite orbit, minimal action, exceptional minimal set) is the same as the one of $G^+$.

The case of a minimal action further splits in three subcases as follows. Recall that a minimal action of $G$ on the circle is \textbf{proximal} (or contracting) if for all open intervals $I, J \subsetneq \Sbb^1$, $J \neq \emptyset$, there exists $g\in G$ such that $g(I)\subset J$. The following result is  a reinterpretation due to Ghys \cite{Ghys-circ} of a result of Margulis \cite{Mar-Tits} (see \S 5.2 in \cite{Ghys-circ}). For $G\le \homeo^+ (\Sbb^1)$, we denote by $\aut_{G}(\Sbb^1)$ the centralizer of $G$ in $\homeo^+(\Sbb^1)$.

\begin{thm} \label{t-Margulis}
Assume that $G\le \homeo^+ (\Sbb^1)$ acts minimally on $\Sbb^1$. Then one of the following holds:

\begin{enumerate}
\item The group $\aut_{G}(\Sbb^1)$ is infinite and  $G$ is abelian and conjugate to a subgroup of the group of rotations.
\item The group $\aut_{G}(\Sbb^1)$ is finite cyclic, and the action of $G$ on the topological circle  $\aut_{G}(\Sbb^1) \backslash \Sbb^1$ is proximal.
\end{enumerate}
In particular the action of $G$ on $\Sbb^1$ is proximal if and only if $\aut_{G}(\Sbb^1)$ is trivial.
\end{thm}

In the course of the proof, we will also use an analogue of Theorem \ref{t-Margulis} for groups of homeomorphisms of the real line. A minimal action of a group $G$ on $\R$ is said to be proximal if for all relatively compact open intervals $I, J\subset \R$, $J \neq \emptyset$, there exists $g\in G$ such that $g(I)\subset J$.  The following result is  \cite[Th.1]{Mal-class} (see the \enquote{Remark on centralizers} for this formulation). As before we denote by $\aut_G(\R)$ the centralizer of $G$ in $\homeo^+(\R)$.

\begin{thm} \label{t-Mar-line}
Let $G\le \homeo^+(\R)$ be a group of orientation preserving homeomorphisms of the real line. Assume that $G$ acts minimally on $\R$. Then exactly one of the following holds:
\begin{enumerate}
\item The group $G$ is abelian and conjugate to a subgroup of the group of translations.
\item $\aut_G(\R)$ is cyclic and generated by an element conjugate to a translation, and the action of $G$ on the topological circle $\aut_G(\R)\backslash \R$ is proximal.
\item The action of $G$ on $\R$ is proximal.

\end{enumerate}
\end{thm}

\subsection{Proofs of Theorems \ref{thm-intro-circle} and \ref{thm-bis-intro-circle}}
 We fix $G \leq \homeo(\Sbb^1)$ and consider a faithful action of $G$ on a set $\Omega$. In order to avoid any confusion between the action of $G$ on $\Sbb^1$ and the action on $\Omega$, points of $\Sbb^1$ will be denoted by $x, y, z, ...$, while points in $\Omega$ will be denoted $\omega, \omega_1, \omega_2, ...$. The notation $G_x$ and $G_\omega$ refer to the stabilizer in $G$ with respect to the corresponding action.


Given an interval $I\subset \Sbb^1$, we denote by $G_I$ the subgroup of $G$ consisting of elements that fix pointwise the complement of $I$. Note that the action of $G$ on $\Sbb^1$ is not topologically free if and only if there exists a closed interval $I \subsetneq \Sbb^1$ such that $G_I$ is non-trivial. If in addition $G$ is minimal and proximal, this is equivalent to the fact that $G_J$ is non-trivial for every non-empty open interval $J$. 





We begin with the following first classification.

 \begin{prop} \label{prop-S1}
Let $G \leq \homeo(\Sbb^1)$ acting minimally and proximally on $\Sbb^1$. Assume that $G$ acts faithfully and 3-transitively on a set $\Omega$. Then exactly one of the following holds:

\begin{enumerate}
\item \label{i-conjugate} For every $\omega\in \Omega$, the group $G_\omega$ fixes a unique point $z(\omega)\in \Sbb^1$, and the map $\omega \mapsto z(\omega)$ conjugates the $G$-action on $\Omega$ to its action on an orbit in $\Sbb^1$. 
\item \label{i-minimal-not-proximal} For every $\omega \in \Omega$, the group $G_\omega$ acts minimally, but not proximally on $\Sbb^1$. Moreover in this case for every $\Delta\in \Omega^{\{ 2\}}$ the action of $G_\Delta$ on $\Sbb^1$ is not minimal.
\item \label{i-minimal-proximal} For every $\omega\in \Omega$, the group $G_\omega$ acts minimally and proximally on $\Sbb^1$.

\end{enumerate}
Moreover if $G$ preserves the orientation then \eqref{i-minimal-proximal} always holds, and for every $\Delta\in \Omega^{\{ 2\}}$ the action of $G_\Delta$ on $\Sbb^1$ is minimal.

\end{prop}

\begin{proof}
The action of $G$ on $\Omega$ is 2-transitive and hence primitive, so $G_\omega$ is a maximal subgroup of $G$. Assume that $G_\omega$ has  unique fixed point $z(\omega)\in \Sbb^1$. It must then be equal to the stabilizer of $z(\omega)$, and we deduce that the $G$-action on $\Omega$ is conjugate to the action on the orbit of $z(\omega)$, and we are in case \eqref{i-conjugate}.

Assume now that case \eqref{i-conjugate} does not hold. Let us first show that $G_\omega$ must be minimal. Assume by contradiction that $K\subset \Sbb^1$ is a proper closed $G_\omega$-invariant subset. Since we are not in case \eqref{i-conjugate}, we can assume that $K$ contains more than one point. Again by maximality, $G_\omega$  must be equal to the setwise stabilizer of $K$ in $G$. It follows that the $G$-action on $\Omega$ is conjugate to the $G$-action on the orbit of $K$. 
In particular, $G_\omega$ must act 2-transitively on the collection $\{g(K)\colon g\in G\}\setminus \{K\}$.  
Assume first that the complement of $K$ contains at least two distinct connected components $I, J$. By proximality of the $G$-action we can find $g, h, f\in G$ such that $g(K), h(K)$ are distinct and contained in $I$, while $f(K)$ is contained in $J$. Since $G_\omega$ preserves $K$, it cannot map the pair $(g(K), h(K))$ to  $(g(K), f(K))$, reaching a contradiction. If the complement of $K$ has a unique connected component $I$, then $K$ is itself a compact interval of non-empty interior. Thus we can find $g, h\in G$ such that $g(K)\subset I$ and $h(K) \subsetneq K$. Again, no element of $G_\omega$ can map $g(K)$ to $h(K)$.  This shows that $G_\omega$ is minimal.

If $G_\omega$ is proximal, then we are in case \eqref{i-minimal-proximal}. Henceforth we assume that $G_\omega$ is not proximal, and we will find $\Delta\in \Omega^{\{ 2\}}$ such  that $G_\Delta$ does not act minimally on the circle (since the groups $G_\Delta$ are pairwise conjugate, this conclusion will automatically hold for every $\Delta$). Since $G_\omega$ admits a 2-transitive faithful action, its index 2 subgroup $G_\omega^+$ is not abelian and thus cannot be conjugate to a group of rotations. By Theorem \ref{t-Margulis} we deduce that $G^+_\omega$ must centralize  a non-trivial element  $c\in \homeo(\Sbb^1)$ of finite order. Upon replacing $c$ with a power we  assume that $c$ is a conjugate to a rotation of an angle $2\pi/n$ for some $n\in \N$.  Let $x_0\in \Sbb^1$ and $ x_i=c^i(x_0)$ for $i=1,\ldots, n-1$. Note that $x_0,\ldots , x_{n-1}$ are cyclically ordered and that each interval $[x_i, x_{i+1})$ ($i$ mod $n$) is a fundamental domain for $c$. By proximality we can choose $g\in G$ such that the points $x_i':=g(x_i)$ all belong to the interval $[x_0, x_1)$ and are cyclically ordered as $x_0< x'_0<\cdots<x'_{n-1}<x_1$. 
Note that $c^g$ centralizes $G_{g(\omega)}^+$ and therefore both $c$ and $c^g$ centralize $G_\Delta^+$ for $\Delta:=\{\omega, g(\omega)\}$. 
In particular, so does their product $c^g c$. Now note that $c^g c([x_0, x_1) )=c^g([x_1, x_2)) \subset c^g([x'_{n-1}, x'_0))= [x'_0, x'_1)$ is a strict subinterval of $[x_0, x_1)$. In particular $c^gc$ must admit fixed points in this interval. It follows that $G^+_\Delta$ must preserve the (closed) set of fixed points of $c^gc$ and thus does not act minimally on $\Sbb^1$. Therefore $G_\Delta$ does not act minimally either, and we are in case \eqref{i-minimal-not-proximal}. This concludes the proof of the first part of the proposition.

Assume now that $G$ is orientation preserving. Then case \eqref{i-conjugate} cannot arise, since the action of $G$ on any of its orbit in $\Sbb^1$ preserves the cyclic order and thus is never $3$-transitive. Let us show that $G_\omega$ is necessarily proximal. By contradiction, let again $c\in \homeo_+(\Sbb^1)$ be a finite order element which centralizes $G_\omega$ and is conjugate to a rotation of angle $2\pi/n$, and $x_i$ be points as above. Since $G$ is proximal, it does not centralize $c$, and thus by maximality  we deduce that $G_\omega$ is precisely the centralizer of $c$ in $G$. We therefore have an injective equivariant map $\Omega\simeq G/G_\omega\to \homeo(\Sbb^1), gG_\omega\mapsto gcg^{-1}$, and we deduce that the conjugation action of $G$ on the conjugacy class of $c$ is 3-transitive. Using proximality and minimality of the action of $G$  we can find $g, h\in G$  such that the points $x'_i=g(x_i)$ and $x''_i=h(x_i)$ lie in $[x_0, x_1)$ and are ordered cyclically as  \[x_0< x'_0<\cdots<x'_{n-1}< x''_0<\cdots < x''_{n-1}< x_1.\] Set $c'=c^g$ and $c''=c^h$, so that the points $x'_i$ form a $c'$-orbit and the points $x''_i$ form an $c''$-orbit. We claim that $G$ cannot map $(c, c', c'')$ to $(c, c'', c')$. To see this, we assume that $k\in G$ centralizes $c$ and satisfies $(c')^k=c'', (c'')^k=c'$, and we analyse the relative position of the points \[k(x_0)< k(x'_0)<\cdots< k(x'_{n-1})< k(x''_0)<\cdots< k(x''_{n-1})< k(x_1)\] with respect to the points \[x_0< x'_0<\cdots<x'_{n-1}< x''_0<\cdots < x''_{n-1}< x_1.\]  We have $c'(k(x''_i))=k(x''_{i+1})$ and  $c''(k(x'_i))=k(x'_{i+1})$  (mod $n$).
 Since $[x'_0, x'_1)$ is a fundamental domain for $c'$, we deduce that there exists $i$ such that $k(x_i'')\in [x'_0, x'_1)$, and the same argument applied to $c''$ shows that there exists $j$ such that $k(x_j')\in [x_0'', x_1'')$. Permuting cyclically the inequality $k(x_0)< k(x'_j)<k(x''_i)< k(x_1)$  we see that this implies that 
\[x_0<x''_0\leq k(x''_i)<k(x_1)<k(x_0)<k(x''_i)< x'_1<x_1. \] This is impossible, since $ck(x_0)=k(x_1)$, and thus $k(x_0)$ and $k(x_1)$ cannot both lie in the fundamental domain $[x_0, x_1)$. This is the desired contradiction, and thus $G_\omega$ must be proximal.

It remains to be shown that for $\Delta=\{\omega_1, \omega_2\} \in \Omega^{\{ 2\}}$ the group $G_\Delta$ acts minimally. Assume that this is not the case. Then either it has an exceptional minimal set, or it admits periodic orbits. We let $K(\omega_1, \omega_2)$ be the exceptional minimal set in the first case, and the set of periodic orbits in the second case. In both cases $K(\omega_1, \omega_2)$ is a $G_\Delta$-invariant closed proper subset of $\Sbb^1$ (in the finite orbit case, this follows from the observation that all finite orbits must have the same cardinality, and that $G_\Delta$ is transitive on $\Omega\setminus \Delta$ and hence infinite). If we denote by $\Omega^{(3)}$ the set of ordered triples of distinct elements of $\Omega$ and by $\mathcal{C}(\Sbb^1)$ the set of closed subsets of $\Sbb^1$, we have a $G$-equivariant map \[ \Phi\colon \Omega^{(3)}\to \mathcal{C}(\Sbb^1)^3, \, \, (\omega_1, \omega_2, \omega_3)\mapsto \left( K(\omega_1,\omega_2), K(\omega_2, \omega_3), K(\omega_1, \omega_3) \right). \] By $3$-transitivity, the action of $G$ on the image of $\Phi$ is transitive. We claim that whenever $(K_1, K_2, K_3)$ is in the image of $\Phi$, each  $K_i$ must be entirely contained in a connected component of  the complement of $K_j$ for every choice of $i\neq j$. By transitivity of the action of $G$ on $\operatorname{Im}(\Phi)$, it is enough to find $(K_1, K_2, K_3)\in \operatorname{Im}(\Phi)$ with this property. Given $\omega, \omega' \in \Omega$, by minimality and proximality of $G_\omega$ we can choose $g\in G_\omega$ such that $g(K(\omega, \omega'))$ is entirely contained in a connected component of the complement of $K(\omega, \omega')$. Setting $\omega''=g(\omega')$, we see that the triple $(K_1, K_2, K_3)=(K(\omega, \omega'), K(\omega, \omega''), K(\omega', \omega''))$ satisfies the claim for $i=1, j=2$ and for $i=2, j=1$. But any $h\in G$ which permutes cyclically the points $\omega, \omega', \omega''$ must permute cyclically the sets $K_1, K_2, K_3$, which implies the same property for all $i, j$. This proves the claim. 


We now observe that this implies that for every $(K_1, K_2, K_3)\in \operatorname{Im}(\Phi)$, each $K_i$ is entirely contained in a connected component of the complement of the union of the other two. To see this, note that if $K_2, K_3$ lie in the same connected component of the complement of $K_1$, this property holds for $i=1$, and if not it holds for $i=2$. Using again that $ \operatorname{Im}(\Phi)$ is stable under cyclic permutations, the property must hold for every $i$.

We deduce that the cyclic order on the circle induces a well-defined cyclic order on triples $(K_1, K_2, K_3)\in \operatorname{Im}(\Phi)$. Now chose $\omega_1, \omega_2, \omega_3\in \Omega$ and $g\in G$ such that $g$ fixes $\omega_1$ and $g$ exchanges $\omega_2$ and $\omega_3$. This element $g$ sends $(K_1, K_2, K_3):=\left(K(\omega_1, \omega_2), K(\omega_2, \omega_3), K(\omega_1, \omega_3)\right)$ to $(K_3, K_2, K_1)$, and thus cannot preserve the cyclic orders the sets $K_i$, reaching a contradiction. Thus $G_\Delta$ must act minimally on $\Sbb^1$, and the proof is complete.    \qedhere
\end{proof}

Note that in Proposition \ref{prop-S1} we did not require that the action of $G$ in $\Sbb^1$ is not topologically free. However in the sequel we restrict to this situation, and we establish preliminary results for the proofs of Theorems \ref{thm-intro-circle} and \ref{thm-bis-intro-circle}. Our first goal is to exclude the existence of a 3-transitive faithful action of $G$ on $\Omega$ with the property that for  $\Delta \in \Omega^{\{2\}}$, the group $G_\Delta$ acts minimally on $\Sbb^1$. This is the purpose of the following statements until Proposition \ref{p-Gdelta-not-minimal}.


\begin{lem} \label{lem-stab-NTF}
Let $G \leq \homeo(\Sbb^1)$ whose action on $\Sbb^1$ is minimal and not topologically free. Suppose that $G$ acts faithfully on a set $\Omega$ such that for any  $\Delta \in \Omega^{\{ 2\}}$, the action of $G_\Delta$ on $\Sbb^1$ is minimal. Then there exists $\omega \in \Omega$ such that the action of $G_{\omega}$ on $\Sbb^1$ is not topologically free.
\end{lem}

\begin{proof}
Choose a non-empty open interval $I$ and a non-trivial $\gamma \in G$ that is supported outside $I$. Let $\omega_1 \in \Omega$ be such that $\omega_2 := \gamma(\omega_1)\neq \omega_1$, and  $\Delta=\{\omega_1, \omega_2\}$. Let $X$ be the set of $g \in G_\Delta$ such that $g I \cap I \neq \varnothing$. By minimality of the action of $G_\Delta$ on $\Sbb^1$, we have $G_\Delta I = \Sbb^1$. So by connectedness, the subgroup $G_\Delta$ is generated by $X$ \cite[Th. 8.10]{Bri-Hae}, and hence $\gamma$ cannot commute with all the elements of $X$. If $g \in G_\Delta$ is such that $J := g I \cap I \neq \varnothing$ and $g' := [g,\gamma]$ is non-trivial, then the element $g'$ acts trivially on the open interval $J$, and $g'$ fixes the point $\omega_2$ in $\Omega$. 
\end{proof}

\begin{lem}\label{l-Gdelta-proximal}
Let $G\le \homeo(\Sbb^1)$ whose action on $\Sbb^1$ is minimal, proximal, and not topologically free. Suppose that $G$ acts faithfully and 3-transitively on a set $\Omega$ such that for  $\Delta\in \Omega^{\{ 2\}}$, the action of $G_\Delta$ on $\Sbb^1$ is minimal. Then the action of $G_\Delta$ on $\Sbb^1$ is proximal.
\end{lem}

\begin{proof}

Consider the three cases in Proposition \ref{prop-S1}. Since $G_\Delta\le G_\omega$ for $\omega \in \Delta$, case \eqref{i-conjugate} cannot happen and case \eqref{i-minimal-not-proximal} cannot happen either. Thus, we deduce that $G_\omega$ acts minimally and proximally on $\Sbb^1$ for every $\omega\in \Omega$.

Let $g$ be a non-trivial element of $G_I$ for some open interval $I \subsetneq \Sbb^1$. Let $\omega\in \Omega$ be a point such that $g(\omega)\neq \omega$, and let $\Delta=\{\omega, g(\omega)\}$. Assume by contradiction that the group $G_\Delta$ is not proximal.  By Theorem \ref{t-Margulis}, there exists a non-trivial $f\in \aut_{\Sbb^1}(G^{+}_{\Delta})$ which is conjugate to a rotation. Choose a non-empty open interval $J$ disjoint from $I$ and small enough so that $f(J) \cap J=\varnothing$. Since  $G_\omega$ is proximal and $G_\omega$ is not topologically free by Lemma \ref{lem-stab-NTF}, there exists a non-trivial element $\gamma\in G_\omega$ supported in $J$. The elements $\gamma$ and $g$ have disjoint support in the circle, and thus commute. We deduce that $\gamma g(\omega)=g(\omega)$, and thus $\gamma\in G_\Delta$. But since $f(J)\cap J=\varnothing$, the element $\gamma$ cannot commute with $f$, reaching a contradiction. \qedhere

\end{proof}

Lemma \ref{lem-commutations-S1} and Proposition \ref{prop--S1-mixed-id} below should be compared respectively with Lemma \ref{lem-mix-id} and Proposition \ref{prop--tree-mixed-id}.

\begin{lem} \label{lem-commutations-S1}
Let $G \leq \homeo(\Sbb^1)$, let $I_1,I_2,I_3$ three disjoint intervals, and let $g_1,g_2,g_3 \in G$ such that $g_i$ is supported inside $I_i$. For $g \in G$, at least one of the following hold:
\begin{enumerate}
	\item \label{item-one-commut-S1} $[g_1^g,g_1]=1 \vee [g_1^g,g_2]=1 \vee [g_1^g,g_3]=1$;
	\item \label{item-two-commut-S1} $[g_2^g,g_1] = [g_3^g,g_1] =1 \vee [g_2^g,g_2] = [g_3^g,g_2] =1 \vee [g_2^g,g_3] = [g_3^g,g_3] =1$.
\end{enumerate}
In particular $G$ satisfies \begin{equation} \label{eq-identity} \forall g \, \, [g_1^g, g_1] = 1 \vee [g_1^g, g_2] = 1 \vee \ldots \vee [g_3^g, g_3] = 1. \tag{$\star$} \end{equation}
\end{lem}

\begin{proof}
If there is $i$ such that $I_i \cap g(I_1) = \varnothing$, then (\ref{item-one-commut-S1}) holds. If there is no such $i$ then $g(I_1)$ intersects the three intervals $I_1,I_2,I_3$, and it follows that $g(I_1)$ must contain one of them. If $i$ is such that $I_i \subset g(I_1)$, then $I_i$ is disjoint from $g(I_2)$ and $g(I_3)$, so that $[g_2^g,g_i] = [g_3^g,g_i] =1$. Hence (\ref{item-two-commut-S1}) holds.
\end{proof}

\begin{prop} \label{prop--S1-mixed-id}
Let $G \leq \homeo(\Sbb^1)$ whose action on $\Sbb^1$ is minimal, proximal and not topologically free. Then $G$ satisfies a non-trivial mixed-identity.
\end{prop}

\begin{proof}
The assumptions imply that we can choose non-trivial elements $g_1,g_2,g_3$ of $G$ that are supported in disjoint intervals, and the rest of the proof is exactly the same as Proposition \ref{prop--tree-mixed-id}, except that Lemma \ref{lem-mix-id} is replaced by Lemma \ref{lem-commutations-S1}.
\end{proof}

We have reached a first step in the proof of the theorems.

\begin{prop}\label{p-Gdelta-not-minimal}
Suppose that $G\le \homeo(\Sbb^1)$ is minimal, proximal, and not topologically free. Then there does not exist a set $\Omega$ on which $G$ acts faithfully and $3$-transitively and such that for $\Delta \in \Omega^{\{ 2\}}$, the action of $G_\Delta$ on $\Sbb^1$ is minimal.
\end{prop}


\begin{proof}
Assume for a contradiction that such an action exists. By Lemma \ref{l-Gdelta-proximal}, the action of $G_\Delta$ on $\Sbb^1$ is proximal. Choose an interval $I$ and a non-trivial element $h\in G_I$. Since $h$ has infinite order, we can find $\omega_1\in \Omega$ such that $\omega_2:=h(\omega_1)\neq \omega_1$ and $\omega_3:=h(\omega_2)\neq \omega_1$.  Let $\Delta'=\{\omega_1, \omega_2\}$. Using that $G_{\Delta'}$ is minimal and proximal, we can choose $\gamma, \gamma' \in G_{\Delta'}$ such that the intervals $I_1=I, I_2=\gamma(I), I_3=\gamma'(I)$ are disjoint.  The elements $h_1=h, h_2=h^\gamma, h_3=h^{\gamma'}$   verify $h_i(\omega_1)=\omega_2$ for every $i=1, 2,3$. Moreover these elements have disjoint support for their action on the circle, and hence commute. It follows that $h_i(\omega_2)=h_1h(\omega_1)=h(\omega_2)=\omega_3$ for every $i=1,2,3$. Now choose $g\in G_\Delta$ such that $g(\omega_3)\neq \omega_3$. By Lemma \ref{lem-commutations-S1}, there exist $i, j$ such that $h_i^g$ and $h_j$ commute. Since $g\in G_\Delta$, we  must have $h_i^g(\omega_1)=\omega_2$ and $h_i^g(\omega_2)=g(\omega_3)$. Again since $h_i^g$ and $h_j$ commute and both send $\omega_1$ to $\omega_2$, we must have $h_i^g(\omega_2)=\omega_3$, and hence $g(\omega_3) = \omega_3$. This is a contradiction. \qedhere
\end{proof}



The results that have been established so far allow to prove Theorem \ref{thm-bis-intro-circle}:

\begin{proof}[Proof of Theorem \ref{thm-bis-intro-circle}]
Let $G$ be as in the statement, and assume first that $G$ preserves the orientation on $\Sbb^1$. Towards a  contradiction, assume that $G$ acts faithfully and 3-transitively on a set $\Omega$. According to Proposition \ref{prop-S1}, for every $\Delta\in \Omega^{\{ 2\}}$ the group $G_\Delta$ must act minimally on the circle. But Proposition \ref{p-Gdelta-not-minimal} exactly says that such a situation cannot happen. Therefore we have a contradiction, and the transitivity degree of $G$ is at most $2$.

Assume now that $G$ does not preserve the orientation on $\Sbb^1$. Suppose that $G$ acts faithfully and 4-transitively on a set $\Omega$. Then for $\Delta\in \Omega^{\{2\}}$, the action of the group $G_\Delta$ on $\Omega\setminus \Delta$ is 2-transitive, and hence primitive. Therefore its non-trivial normal subgroup $G_\Delta^+$ acts transitively on $\Omega\setminus \Delta$. Since this holds for arbitrary $\Delta$ and since $\Omega$ is infinite, this implies that the action of $G^+$ on $\Omega$ is $3$-transitive. This contradicts the previous paragraph. Therefore such an action does not exist, and the transitivity degree of $G$ is at most $3$.
\end{proof}



The proof of Theorem \ref{thm-intro-circle} will require a more detailed discussion depending on the nature of the group $G_\Delta$. 

\begin{lem} \label{lem-stab-trans-ome}
Let $G \leq \homeo(\Sbb^1)$ whose action on $\Sbb^1$ is minimal, proximal and not topologically free. Suppose that $G$ acts $3$-transitively on a set $\Omega$ with the property that for $\omega\in \Omega$, the group $G_\omega$ acts minimally on $\Sbb^1$; and for $\Delta \in \Omega^{\{ 2\}}$, the group $G_\Delta$ admits a proper closed invariant subset $C_\Delta \subsetneq \Sbb^1$ that is preserved by $G_{(\Delta)}$.  Then the stabilizer of $C_\Delta$ in $G$ acts transitively on $\Omega$.
\end{lem}

\begin{proof}
In all the proof $\Delta$ is fixed, and for simplicity we write $C = C_\Delta$ and we denote by $S$ the stabilizer of $C$ in $G$. We want to show that $S$ acts transitively on $\Omega$.

Note that $G_{(\Delta)}\subset S$ by assumption, and that  by 3-transitivity the group $G_{(\Delta)}$ has two orbits on $\Omega$, namely $\Delta$ and its complement. Hence $S$ acts transitively on $\Omega$ if and only if $S$ contains $G_{(\Delta)}$ as a proper subgroup. Assume for a contradiction that $S =G_{(\Delta)}$. Given an open interval $I$ that does not intersect $C$, we clearly have $G_I \leq S$, so $G_I \leq G_{(\Delta)}$. In particular we have that the commutator subgroup $[G_I,G_I]$ lies inside $G_{\omega}$ for $\omega \in \Delta$. This implies that $G^+_{\omega}$ (hence $G_\omega$) is proximal. Otherwise, by Theorem \ref{t-Margulis} there exists a non-trivial $f$ conjugate to a rotation which centralizes $G_\omega^+$. Clearly we can choose $I$ as above small enough so that $f(I)\cap I=\varnothing$, and we see that $f$ cannot centralize $[G_I, G_I]$, reaching a contradiction. Thus $G_{\omega}$ is  proximal. Hence given an arbitrary proper closed interval $J$ in $\Sbb^1$, we may find $g \in G_{\omega}$ such that $g(J) \subseteq I$. In particular we have $gG_Jg^{-1} \leq G_I$ and $g[G_J,G_J]g^{-1} \leq [G_I,G_I] \leq G_{\omega}$. Since $g$ is in $G_{\omega}$ we deduce that $[G_J,G_J] \leq G_{\omega}$. The subgroup $N$ of $G$ generated by all these $[G_J,G_J]$ is therefore a non-trivial normal subgroup of $G$ that is contained in $G_{\omega}$. By transitivity we would deduce that $N$ is contained in $G_{\omega'}$ for every $\omega'\in \Omega$, and thus acts trivially on $\Omega$, contradicting that the action is faithful. \qedhere \end{proof}

The following statement is analogous to Proposition \ref{prop-trees-GDelta-focal}, and the proof follows essentially the same lines. We will refer to the proof of Proposition \ref{prop-trees-GDelta-focal}, and explain the modifications that are needed to adapt the arguments to the present setting.

\begin{prop} \label{prop-st2-ptfix-contr}
Let $G\le\homeo(\Sbb^1)$ whose action on $\Sbb^1$ is minimal, proximal, and not topologically free. Then  $G$ does not admit any faithful 3-transitive action on a set $\Omega$ with the following properties:
\begin{enumerate}
\item For  every $\omega\in \Omega$ the group $G_\omega$ acts minimally and proximally on $\Sbb^1$.
\item For every $\Delta\in \Omega^{\{ 2\}}$ the action of $G_\Delta$ on $\Sbb^1$ has a unique fixed point $z(\Delta)$, and the action of $G_\Delta$ on $\Sbb^1\setminus \{z(\Delta)\}$ is minimal.
\end{enumerate}
\end{prop}

\begin{proof}
We let $\pi\colon \Omega^{\{ 2\}}\to \Sbb^1$ be the map $\Delta \mapsto z(\Delta)$. Since $G_\Delta$ has index 2 in $G_{(\Delta)}$, we deduce that  $z(\Delta)\in \Sbb^1$ is also  fixed by $G_{(\Delta)}$. For simplicity we denote by $S_\Delta$ the stabilizer of $z(\Delta)$ in $G$. We also denote by $S_{\Delta}^0$ the subgroup of $S_{\Delta}$ consisting of elements having trivial germs around the point $z(\Delta)$, i.e.\ elements $g \in S_{\Delta}$ such that there exists an open interval $I \subset \Sbb^1$ such that $z(\Delta) \in I$ and $g$ acts trivially on $I$. Note that $S_\Delta^0$ consists of orientation-preserving homeomorphisms.

Exactly as in the proof of Proposition \ref{prop-trees-GDelta-focal}, one verifies the following properties:
\begin{enumerate}[label=(\alph*)]
\item \label{i-S-transitive} For $\Delta \in \Omega^{\{ 2\}}$, we have $G_{(\Delta)}\neq S_\Delta$, and the group $S_\Delta$ acts transitively on $\Omega$.
\item  \label{i-S-2transitive} For every $\Delta \in \Omega^{\{ 2\}}$, the fiber $\pi^{-1}(z(\Delta))$ forms a partition $\mathcal{P}_\Delta$ of $\Omega$ into blocks of size $2$, and $S_{\Delta}$ preserves $\mathcal{P}_\Delta$ and acts $2$-transitively on its blocks.
\item  \label{i-S-germs} For every $\Delta \in \Omega^{\{ 2\}}$, we have $S_{\Delta}^0 G_{\Delta} = S_{\Delta}$.

\end{enumerate}
The proof of \ref{i-S-transitive} is the same as Lemma \ref{lem-tree-Strans}, except that \textit{half-tree} is replaced by \textit{non-empty open interval}. The proof of \ref{i-S-2transitive} is identical to the proof of Lemma  \ref{lem-tree-SDelta-2tran}, and the proof of \ref{i-S-germs} is identical to the proof of Lemma \ref{lem-tree-S-germs}. 


The following is the analogue of Lemma \ref{lem-tree-loc-exch}. However here the proof differs in our situation.


\begin{lem}\label{lem-circle-loc-exch}
For every non-empty open interval $I\subset \Sbb^{1}$, there exists $g\in G_I$ which preserves $\Delta$ and exchanges the two elements of $\Delta$. 
\end{lem}

\begin{proof}
Let $g\in G_{(\Delta)}\le S_\Delta$ be an element which exchanges the two elements of $\Delta$. By property \ref{i-S-germs} above, we have $g=g'g''$ for some $g'\in S_{\Delta}^0$ and $g''\in G_\Delta$. Since $g''$ fixes $\Delta$ it follows that $g'$ must also exchange the two elements of $\Delta$. Thus, there exist a closed interval $J$ not containing $z(\Delta)$ and an element $g'\in G_J$ which exchanges the two elements of $\Delta$. We now want to upgrade this conclusion to show that this holds for every interval $I$.

The same argument also shows that the intersection $S^0_\Delta \cap G_{(\Delta)}$ is non-trivial. Since $S^0_\Delta$ is torsion free, this implies that $S^0
_\Delta \cap G_{(\Delta)}$ is infinite, and thus its index 2 subgroup $S^0
_\Delta \cap G_{\Delta}$ is also non-trivial. In other words, if we see   $G_\Delta^+$ as a group of homeomorphisms of $\R\simeq \Sbb^{1} \setminus\{z(\Delta)\}$, then $G_\Delta^+$ must contain elements of compact support, and thus $G_\Delta^+$ cannot be centralized by any element of $\homeo^+(\R)$ conjugate to a translation. Since $G_\Delta^+$ acts minimally on $\Sbb^{1} \setminus\{z(\Delta)\}$, Theorem \ref{t-Mar-line} therefore implies that the action of $G_\Delta^+$ on $\Sbb^1\setminus \{z(\Delta)\}$ is proximal. Hence if $I\subset \Sbb^1$ is a non-empty open interval, we can find $h\in G_\Delta$ such that $h(J)\subset I$, and consequently the element $hg'h^{-1}$  belongs to $G_I$ and exchanges the two elements of $\Delta$. 
 \qedhere
\end{proof}

Using these ingredients, the end of the proof  is similar as in Proposition \ref{prop-trees-GDelta-focal}. More precisely, we fix $\Delta \in \Omega^{\{2\}}$ and $\Delta'\in \mathcal{P}_\Delta$ distinct from $\Delta$. Reasoning as in the proof of Lemma \ref{lem-tree-loc-3cyc}, we can find a strict interval $I\subset \Sbb^1$ not containing $z(\Delta)$ and $s\in G_I$ such that $s(\Delta)=\Delta'$ and $s(\Delta')\neq \Delta$. Choose $\Lambda\in \Omega^{\{2\}}$ such that $\Lambda\cap \Delta$ and $\Lambda \cap \Delta'$ are non-empty. Choose a non-empty open interval $J$ disjoint from $I$ and which contains neither $z(\Delta)$ nor $z(\Lambda)$. We apply Lemma \ref{lem-circle-loc-exch} to the point $z(\Lambda)$ and find $t\in G_J$ which exchanges the two elements of $\Lambda$. Exactly as in Proposition \ref{prop-trees-GDelta-focal}, the element $t$ preserves the partition $\mathcal{P}_\Delta$ and hence $t$ exchanges the blocks $\Delta, \Delta'$. Since $t$ and $s$ commute by construction, we obtain a contradiction with the fact that $s(\Delta')$ is distinct from $\Delta$.
\end{proof}

We note that the following result is the only place in the proof of Theorem \ref{thm-intro-circle} where we need the assumption that distinct points of the circle have distinct stabilizers in $G^+$.

\begin{prop} \label{prop-st2-min-ptfix}
Let $G \leq \homeo(\Sbb^1)$ whose action on $\Sbb^1$ is minimal, proximal and not topologically free, and assume that two distinct points of the circle have distinct stabilizers in $G^+$. Suppose that $G$ acts $3$-transitively on a set $\Omega$ such that point stabilizers have a minimal action on the circle. Then the following hold:
\begin{enumerate} 
\item For  $\Delta\in \Omega^{\{ 2\}}$ the group $G_{\Delta}$ has a unique fixed point $z \in \Sbb^1$, and the action of $G_\Delta$ on $\Sbb^1\setminus \{z\}$ is minimal. 

\item For  $\omega \in \Omega$ the group $G_{\omega}$  acts proximally on $\Sbb^1$. 
\end{enumerate}
\end{prop}

\begin{proof}
By Proposition \ref{p-Gdelta-not-minimal}, we know  that $G_{\Delta}$ does not act minimally. Let $K\subset \Sbb^1$ be a closed proper invariant subset. If $\tau$ is an element of $G_{(\Delta)}$ outside $G_{\Delta}$, we denote $K' = K \cup \tau(K)$, which is a proper subset of the circle that is invariant by $G_{\Delta}$. We denote by $S$ the stabilizer of $K'$ in $G$, so that $G_\Delta\leq S$. For $\omega\in \Delta$ observe that since $G_{\omega}$ is minimal on the circle and by maximality of $G_{\Delta}$ in $G_{\omega}$, we have that $G_{\Delta}$ is exactly the stabilizer of $K'$ in $G_{\omega}$. 

According to Lemma \ref{lem-stab-trans-ome} we have $G_{\omega} S = G$. Hence the map from $G_{\omega} / G_{\Delta}$ to $G/S$ is onto, so that $G_{\omega}$ acts $2$-transitively on the $G$-orbit of $K'$. By proximality of the $G$-action, there are elements $g$ in $G$ that send $K'$ inside a connected component of the complement of $K'$. Since $G_{\omega}$ acts $2$-transitively on the $G$-orbit of $K'$, it follows that every $g(K')$ for some $g \in G$ is either equal to $K'$, or contained in a connected component of the complement of $K'$. 

Now argue by contradiction and assume that $K'$ has cardinality at least two. Let $I$ be a connected component of the complement of $K'$, and let $a,b \in K'$ be the endpoints of $I$. We consider the stabilizer $G_a^+$ of $a$ in $G^+$. Since $G_a^+$ does not send $K'$ disjoint from itself, by the previous paragraph we have that $G_a^+$ stabilizes $K'$. In particular $G_a^+$ must send $I$ to another  connected component, which is necessarily $I$ since the point $a$ is fixed. Hence we deduce that the point $b$ is also fixed by $G_a^+$, so that $a$ and $b$ are distinct points of the circle which have the same stabilizers in $G^+$. By assumption this cannot happen. Hence $K'$ is a singleton, and a fortiori $K$ is a singleton.

We have thus shown that $G_{\Delta}$ has a unique fixed point $z$, and that $\{z\}$ is the unique proper closed subset of $\Sbb^1$ which is invariant under $G_\Delta$. In particular, $G_\Delta$ acts minimally on $\Sbb^1\setminus \{z\}$. 

Finally the action of the group $G_\omega^+$ on $\Sbb^1$ must be proximal. Indeed if $c\in \homeo_+(\Sbb^1)$ is an element of finite order that centralizes $G_\omega^+$, then $c$ also centralizes $G_\Delta^+$ for every $\Delta$ containing $\omega$. Therefore $c$ must fix the point $z(\Delta)$, and we deduce that $c=\operatorname{Id}$.  \qedhere
\end{proof}

We are finally able to complete the proof:

\begin{proof}[Proof of Theorem \ref{thm-intro-circle}]
Suppose that $G\le \homeo(\Sbb^1)$ satisfies the assumptions of the theorem, and assume that $\Omega$ is a set on which $G$ acts faithfully and 3-transitively. By Proposition \ref{prop-S1}, if the action of $G$ on $\Omega$ is not conjugate to the action on an orbit in $\Sbb^1$, then for every $\omega\in \Omega$ the group $G_\omega$ acts minimally on $\Sbb^1$. In this situation we can successively apply Proposition \ref{prop-st2-min-ptfix}  and Proposition \ref{prop-st2-ptfix-contr}, and we reach a contradiction. \qedhere
\end{proof}

\subsection{An example: Thompson's group $T$} \label{subsec:ThompsonT}
Recall that Thompson's group $T$ is the group of orientation preserving homeomorphisms $g$ of $\R/\Z$ that are piecewise linear, with finitely many discontinuity points for the derivative, each being a dyadic rational (i.e. a rational number whose denominator is a power of 2), and such that in restriction to each piece $g$ has the form $x\mapsto 2^nx+q$ with $n\in \Z$ and $q$ a dyadic rational. The group $T^{\pm}$ is the group of homeomorphisms of $\R/\Z$ of the form $x\mapsto \pm g(x)$, with $g\in T$. The group $T^{\pm}$ contains the group $T$ as a subgroup of index $2$.

\begin{proof}[Proof of Corollary \ref{cor-intro-S1-T}]
It is well-known and easy to see that the group $T$ acts 2-transitively on the set of dyadic rationals in  $\R/\Z$, and thus $td(T)\geq 2$. On the other hand the action of the group $T$ on the circle is minimal, proximal and not topologically free, so that $td(T)\le 2$ according to Theorem \ref{thm-bis-intro-circle}. Therefore  $td(T) =  2$.

Similarly the group $T^\pm$ acts 3-transitively on the dyadic rationals, and $td(T^\pm )\leq 3$ by Theorem \ref{thm-bis-intro-circle}, so $td(T^\pm ) = 3$.
\end{proof}


Distinct points in $\R/\Z$ have distinct stabilizers in $T$, so that the group $T^\pm$ satisfies all the assumptions of Theorem \ref{thm-intro-circle}. Therefore the only 3-transitive actions of the group $T^\pm$ are the actions on an orbit in $\R/\Z$. We do not know whether an analogous result holds for the group $T$. 

\begin{question}
Does Thompson's group $T$ admit a 2-transitive action which is not conjugate to the action on an orbit in $\R/\Z$ ?
\end{question}

\section{Groups acting on the real line} \label{sec-line}

In this section we consider groups acting on the real line. Given $G \leq \homeo(\mathbb{R})$, we denote by $G^0$ the compactly supported elements of $G$. For $x \in \mathbb{R}$ we denote by $G_x^-$ and $G_x^+$ the subgroups of the stabilizer $G_x$ that are respectively supported in $(-\infty,x]$ and $[x,+ \infty)$. 

Recall that a permutation group $G \leq \Sym(\Omega)$ is \textbf{regular} if $G$ is transitive and $G_\omega = 1$ for every $\omega \in \Omega$.

\begin{prop} \label{prop-line-orbits-disj}
Suppose that $G \leq \homeo^+(\mathbb{R})$ acts on $\mathbb{R}$ with no global fixed points, and $G^0 \neq 1$. Let $\Omega$ be a set on which $G$ acts faithfully and $2$-transitively. Then at least one of the following happens:
\begin{enumerate}
	\item $G^0$ is regular on $\Omega$;
	\item for every $x \in \mathbb{R}$, for every orbit $O^- \subset \Omega $ of $G_x^-$ and $O^+ \subset \Omega $ of $G_x^+$, we have $|O^- \cap O^+| \leq 1$.
\end{enumerate}
\end{prop}

\begin{proof}
Since $G^0$ is a non-trivial normal subgroup of $G$ and the action of $G$ on $\Omega$ is $2$-transitive, $G^0$ acts transitively on $\Omega$. Fix $x \in \mathbb{R}$, and assume that $G_x^-$ acts transitively on $\Omega$. Let $g \in G^0$, $g \neq 1$, and let $C$ be a compact interval into which $g$ is supported. Since $G$ acts on $\mathbb{R}$ with no fixed points, we may find $h \in G$ such that $h(C) \subset [x,+ \infty)$. It follows that $hgh^{-1} \in G_x^+$, and hence commutes with $G_x^-$. Since $G_x^-$ acts transitively on $\Omega$ and $G_x^-$ commutes with $G_x^+$, we deduce that $hgh^{-1}$ has no fixed point in $\Omega$, and therefore $g$ has no fixed point in $\Omega$. So any non-trivial element of $G^0$ has no fixed point in $\Omega$, and $G^0$ is regular. The argument is the same if $G_x^+$ acts transitively on $\Omega$.

Hence we assume that $G_x^-$ and $G_x^+$ both have at least two orbits in $\Omega$, and we denote these orbits respectively by $(O_i^-)_{i \in I}$ and $(O_j^+)_{j \in J}$. We will show that the second situation holds. We make the observation that for $g \in G$, if $g(x) \leq x$ then $g G_x^- g^{-1} \leq G_x^-$, and hence for every $i \in I$ there exists $i' \in I$ such that $g(O_i^-) \subset O_{i'}^-$. Similarly if $g(x) \geq x$ then $g G_x^+ g^{-1} \leq G_x^+$, and for every $j \in J$ there exists $j' \in J$ such that $g(O_j^+) \subset O_{j'}^+$.

Suppose for a contradiction that there exist $O_i^-$ and $O_j^+$ such that $|O_i^- \cap O_j^+|$ contains at least two points, and fix $\omega_1,\omega_2 \in O_i^- \cap O_j^+$ two distinct points. Suppose that there is $\ell \in J$ such that $O_\ell^+ \neq O_j^+$ and $O_\ell^+$ is not contained in $O_i^-$, and choose $\omega_3 \in O_\ell^+$ such that $\omega_3 \notin O_i^-$. By $2$-transitivity we know that there is $g \in G$ such that $g(\omega_1) = \omega_1$ and $g(\omega_2) = \omega_3$. If $g(x) \leq x$ then by the observation above and the fact that $g$ fixes $\omega_1$, we would have $g(O_i^-) \subset O_i^-$ and hence $\omega_3 \in O_i^-$, a contradiction. Similarly if $g(x) \geq x$ then $g(O_j^+) \subset O_j^+$ and $\omega_3 \in O_j^+$, which is also a contradiction. Hence all possibilities lead to a contradiction, and it follows that every $O_\ell^+$ distinct from $O_j^+$ is contained in $O_i^-$. By the same argument every $O_k^-$ distinct from $O_i^-$ is contained in $O_j^+$. 

Fix $O_\ell^+$ distinct from $O_j^+$ and $O_k^-$ distinct from $O_i^-$, and let $\omega_3 \in O_\ell^+$ and $\omega_4 \in O_k^-$. Again there is $g \in G$ such that $g(\omega_1) = \omega_3$ and $g(\omega_3) = \omega_4$. By a similar argument as before, if $g(x) \leq x$ then $g(O_i^-) \subset O_i^-$ and $\omega_4 \in O_i^-$; and if $g(x) \geq x$ then $g(O_j^+) \subset O_j^+$ and $\omega_3 \in O_j^+$. So in any case we have reached a contradiction, so we have shown that $|O_i^- \cap O_j^+| \geq 2$ is impossible.
\end{proof}

\begin{prop} \label{prop-line-td2}
Suppose that $G \leq \homeo^+(\mathbb{R})$ acts on $\mathbb{R}$ with no fixed points, and $G^0 \neq 1$. Then the transitivity degree of $G$ is at most $2$.
\end{prop}

\begin{proof}
Suppose for a contradiction that $\Omega$ is a set on which $G$ acts faithfully and $3$-transitively, and fix $x \in \mathbb{R}$. We first claim that the group $G_x^+$ does not fix a point in $\Omega$. Otherwise, let $\omega \in \Omega$ that is fixed by $G_x^+$. If $G_\omega$ fixes a point $z \in \mathbb{R}$ then by maximality of $G_\omega$ we would have $G_\omega = G_z$, and the $G$-action on $\Omega$ would be conjugate to the $G$-action on the orbit $G(z)$, which clearly contradicts $3$-transitivity. So $G_\omega$ does not fix a point in $\mathbb{R}$. It follows that every element $g \in G^0$ can be conjugated inside $G_x^+$ with an element of $G_\omega$, and hence that $g$ fixes $\omega$ because $G_x^+$ does. Therefore the normal subgroup $G^0$ fixes the point $\omega$, and hence $G^0$ is trivial by $2$-transitivity. This is a contradiction. So the group $G_x^+$ does not fix a point in $\Omega$, and the same argument applies to $G_x^-$.

We now apply Proposition \ref{prop-line-orbits-disj}. If $G^0$ is regular on $\Omega$, then by $3$-transitivity $G^0$ would be an elementary abelian $2$-group \cite[Theorem 7.2.A]{Dixon-Mortimer}, which is absurd because $G^0$ is torsion free. Therefore whenever $O^-$ and $O^+$ are orbits under $G_x^-$ and $G_x^+$ in $\Omega$, we have $|O^- \cap O^+| \leq 1$. The end of the proof is similar as the argument in Proposition \ref{prop-line-orbits-disj}. Fix $O^-$ and $O^+$ which intersect each other along a singleton $\omega_1$. By the first paragraph $O^-$ and $O^+$ have cardinality at least $2$, so that we may find $\omega_2 \in O^-$ and $\omega_3 \in O^+$ such that $\omega_1,\omega_2,\omega_3$ are all distinct. Since $G$ acts $3$-transitively on $\Omega$, there exists $g \in G$ such that $g(\omega_1) = \omega_1$, $g(\omega_2) = \omega_3$ and $g(\omega_3) = \omega_2$. If $g(x) \leq x$ then $g(O^-) \subset O^-$ and $\omega_3 \in O^-$; and if $g(x) \geq x$ then $g(O^+) \subset O^+$ and $\omega_2 \in O^+$. In both cases we have a contradiction because $\omega_3 \notin O^-$ and $\omega_2 \notin O^+$ by definition. This terminates the proof.
\end{proof}

Hull and Osin asked the question of computing the transitivity degree of Thompson's group $F$ \cite{Hull-Osin}. The following partial answer follows from Proposition \ref{prop-line-td2}:

\begin{cor} \label{cor-line-F}
The transitivity degree of Thompson's group $F$ is at most $2$.
\end{cor}

Recently certain maximal subgroups of $F$ were investigated by Golan and Sapir in \cite{Golan-Sapir-max}. Equivalently, these correspond to primitive actions of $F$. One maximal subgroup exhibited there is the stabilizer of a partition of an orbit of $F$ in the interval \cite{Golan-Sapir-JS}. We do not know whether subgroups of this kind, i.e.\ subgroups $H \leq F$ such that there exists an $F$-orbit $\mathcal{O} \subset [0,1]$ and a partition $\mathcal{P}$ of $\mathcal{O}$ such that $H$ is the stabilizer of $\mathcal{P}$; could give rise to an action of $F$ on $F/H$ that is $2$-transitive.

\newpage 

\appendix

\section{Mixed identities and transitivity degree}

Let $G$ be a group, and $w \in G \ast \mathbb{Z}$. In the sequel by the \textbf{length} of $w$ we mean the word length of $w$ with respect to the generating subset $S = G \cup \left\{t^{\pm 1}\right\}$, where $t$ is a generator of $\mathbb{Z}$. Recall that it is the smallest integer $\ell$ such that there exist $s_1,\ldots,s_\ell \in S$ such that $w = s_1 \cdots s_\ell$. Recall that  $G$ satisfies the mixed-identity $w=1$ if every homomorphism from $G \ast \mathbb{Z}$ to $G$ that is identical of $G$ is trivial on $w$. Replacing $\Z$ with any non-abelian free group $\mathbb{F}_m$ yields an equivalent definition \cite[Remark 5.1]{Hull-Osin}.


Given an infinite set $\Omega$, we denote by $\sym_f(\Omega)<\sym(\Omega)$ the group of permutations with finite support, and by $\mathrm{Alt}_f(\Omega)$ its alternating subgroup. The latter is simple and is the only non-trivial proper normal subgroup of $\sym_f(\Omega)$.

Hull and Osin showed in \cite{Hull-Osin} that if a group $G$ satisfies a non-trivial mixed identity, then $G$ cannot act faithfully and highly transitively on a set $\Omega$, unless $G$ contains the subgroup $\alt_f(\Omega)$.  This raises the natural question whether this is  can be strengthened to obtain that $G$ has finite transitivity degree (this question is discussed after Question 6.2  in \cite{Hull-Osin}). The purpose of this appendix is to prove that this is indeed the case.  

\begin{prop} \label{prop-mixid-td}
Let $G$ be an infinite group that satisfies a non-trivial mixed identity $w\in G\ast \Z$ of length $k$. Then exactly one of the following holds: 

\begin{enumerate}
\item \label{i-alt-normal} There exists a set $\Omega$ and an injective homomorphism $G\to \sym(\Omega)$ whose image contains $\mathrm{Alt}_f(\Omega)$. 
 \item \label{i-bound} We have $td(G)<k$.
 \end{enumerate} 
\end{prop}

\begin{proof}
Let $w\in G\ast \Z $ be a mixed identity of length $k$ that is satisfied by $G$, and let $t$ be a generator of $\Z$. Upon replacing  $w$ by a cyclic conjugate without increasing its length, we can suppose that $w$ is of the form  $w=t^{n_r}g_r\cdots  t^{n_1}g_1$, with $g_i\neq 1$ and $n_i\neq 0$ for $i=1,\ldots, r$, so that the length of $w$ is   $k= r+\sum |n_i|$. We proceed by assuming that $td(G)\geq k$ and showing that \eqref{i-alt-normal} must hold.  To this end, let $\Omega$ be a set on which $G$ acts faithfully $k$-transitively. Assume first that $G$ contains a non-trivial element whose support in $\Omega$  has cardinality which does not exceed $k$. By $k$-transitivity, we deduce that $G$ must contain all conjugates of $g$ in $\sym_f(\Omega)$. In particular, $G$ contains a non-trivial normal subgroup of $\sym_f(\Omega)$, and thus contains $\alt_f(\Omega)$.

Assume now that the support of every non-trivial element  $g
\in G$ has at least $k+1$ points, and let us show that this leads to a contradiction. This assumption applied to the elements $g_i$ shows that we can choose points $\omega_i$ in the support of $g_i$ in such a way that the $2r\leq k$ elements $\omega_1, g_1(\omega_1), \ldots, \omega_r, g_r(\omega_r)$ are distinct. Since $\Omega$ is infinite, for each $i$ it is possible to find $\omega'_{i, 1},\ldots, \omega'_{i, |n_i|} \in \Omega$ such that all the points \[ \omega_1, g_1(\omega_1), \omega'_{1,1},\ldots,  \omega'_{1, |n_1|}, \omega_2, g_2(\omega_2), \omega'_{2,1},\ldots,  \omega'_{2, |n_2|},  \ldots, \omega_r, g_r(\omega_r), \omega'_{r, 1},\ldots, \omega'_{r, |n_r|}\] are  distinct. Using $k$-transitivity along with the fact that $k=\sum_{i=1}^r (|n_i|+1)$, we can find $h\in G$ which verifies the following condition for every $i=1,\ldots, r$:  
\[\left\{\begin{array}{lr}  h\colon (g_i(\omega_i),  \omega'_{i, 1}, \cdots,  \omega'_{i, |n_i|}) \mapsto (  \omega'_{i, 1}, \cdots,  \omega'_{i, |n_i|}, \omega_{i+1}) & \text{if } n_i>0\\ h\colon  (  \omega'_{i, 1}, \cdots,  \omega'_{i, |n_i|}, \omega_{i+1})\mapsto (g_i(\omega_i),  \omega'_{i, 1}, \cdots,  \omega'_{i, |n_i|})   & \text{if } n_i<0 \end{array} \right. \quad \]
Observe that by construction we have $h^{n_i}g_i(\omega_i)= \omega_{i+1}$ for $i=1,\ldots, r-1$ and $h^{n_r}g_r(\omega_r)=\omega'_{r, |n_r|}$. 
Therefore if we let $\overline{w}\in G$ be the element obtained by evaluating $w$ on $t=h$, we have 
\[\overline{w}(\omega_1)=h^{n_r}g_r\cdots h^{n_1}g_1(\omega_1)= \cdots = h^{n_r}g_r\cdots h^{n_{i+1}}g_i(\omega_i)= \cdots =h^{n_r}g_r(\omega_r)=\omega'_{r, |n_r|}.\]
Since $\omega'_{r, |n_r|}\neq \omega_1$, this implies that $\overline{w}\neq 1$ in $G$, reaching a contradiction.

So we have shown that either \eqref{i-alt-normal} or \eqref{i-bound} must hold, and the two  cases are clearly mutually exclusive because in case \eqref{i-alt-normal} the group is highly transitive. 
 \qedhere
\end{proof}

\bibliographystyle{abbrv}
\bibliography{bib-trans-deg}

\end{document}